\documentclass[12pt]{amsart}
\usepackage{cite}
\usepackage{amsfonts}	
\usepackage{amstext}
\usepackage{amsthm,amsmath,amsfonts,amssymb,amscd,latexsym,txfonts,stmaryrd}
\usepackage{color}
\usepackage{epsfig}
\usepackage[all,cmtip]{xy}
\usepackage{enumerate}

\DeclareMathAlphabet{\mathpzc}{OT1}{pzc}{m}{it}

\newcommand{\xyGgraph}{\xymatrixcolsep{1.5pc}\xymatrixrowsep{3pc}\xymatrix}
\newcommand{\xyPoset}{\xymatrixcolsep{.5pc}\xymatrixrowsep{.7pc}\xymatrix}
\newcommand{\xyGraph}{\xymatrixcolsep{1.5pc}\xymatrixrowsep{2.5pc}\xymatrix}

\newcommand{\xyto}{\xymatrixcolsep{1pc}\xymatrix}
\newcommand{\rel}{S^{W_\Theta}_W}
\newcommand{\xymap}{\xymatrixcolsep{5pc}\xymatrixrowsep{1pc}\xymatrix}

\newcommand{\sgn}{\operatorname{sgn}}
\newcommand{\Hom}{\operatorname{Hom}}

\newcommand{\spa}{\operatorname{span}}

\newcommand{\sra}{\shortrightarrow}
\newcommand{\End}{\operatorname{End}}
\newcommand{\Sym}{\operatorname{Sym}}

\newcommand{\C}{{\mathbb C}}
\newcommand{\R}{{\mathbb R}}

\newcommand{\N}{{\mathbb N}}

\theoremstyle{theorem}
\newtheorem{theorem}{Theorem}
\newtheorem*{problem*}{Problem}

\newtheorem*{question*}{Question}

\newtheorem{definition}{Definition}[section]

\newtheorem*{remarks*}{Remarks}
\newtheorem*{claim*}{Claim}

\newtheorem*{remark*}{Remark}

\newtheorem*{hlt*}{Hard Lefschetz Theorem}
\newtheorem*{basisthm*}{Basis Theorem}
\newtheorem*{relbasisthm*}{Relative Basis Theorem}
\newtheorem*{primdecomp*}{Primitive Decomposition Theorem}
\newtheorem{proposition}[definition]{Proposition}
\newtheorem{lemma}[definition]{Lemma}

\newtheorem{corollary}[definition]{Corollary}

\numberwithin{equation}{section}
\begin{document}
\title[Lefschetz property for coinvariant rings]{The strong Lefschetz property for coinvariant rings of finite reflection groups}
\author{Chris McDaniel}
\address{Dept. of Math. and Stat.\\
University of Massachusetts\\
Amherst, MA 01003}
\email{mcdaniel@math.umass.edu}
\begin{abstract}
In this paper we prove that a deformed tensor product of two Lefschetz algebras is a Lefschetz algebra.  We then use this result in conjunction with some basic Schubert calculus to prove that the coinvariant ring of a finite reflection group has the strong Lefschetz property.  
\end{abstract}
\maketitle

\section{Introduction}
In this paper we study the strong Lefschetz properties of the coinvariant ring $S_W$ associated to a finite reflection group $W$.  A commutative graded ring of the form $R=\bigoplus_{i=0}^r R^i$ with $R^0=\R$ is called an \emph{$\N$-graded Artinian $\R$-algebra} and $R$ is said to have the \emph{strong Lefschetz property} if there exists an element $l\in R^1$ such that the multiplication maps $l^{r-2i}\colon R^i\rightarrow R^{r-i}$ are isomorphisms for $0\leq i\leq\lfloor\frac{r}{2}\rfloor$.  The element $l$ is called a \emph{Lefschetz element} and the pair $(R,l)$ is called a \emph{Lefschetz algebra}.

We prove the following result:
\begin{theorem}
\label{thm:coinvlef}
If $W$ is a finite reflection group then $S_W$ has the strong Lefschetz property.
\end{theorem}
The result of Theorem \ref{thm:coinvlef} is not new.  If $W$ is crystallographic, then a classical result of Borel states that $S_W$ is isomorphic to the cohomology ring of an associated flag variety $G/B$.  In this case Theorem \ref{thm:coinvlef} follows from the hard Lefschetz theorem in algebraic geometry.  It should be added that the hard Lefschetz theorem is highly non-trivial:  Hodge proved it over $\C$ using his theory of harmonic integrals (now called Hodge theory) and later Deligne proved it in characteristic $p$ using a version of the Riemann hypothesis for varieties over finite fields; see \cite{M} for more details.  For non-crystallographic $W$, Theorem \ref{thm:coinvlef} has been verified by direct computation in types $I_2(m)$ and $H_3$ by Maeno, Numata and Wachi \cite{MNW} and in type $H_4$ by Numata and Wachi \cite{NW}.  Here also it should be added that the computations for type $H$ in \cite{MNW} and \cite{NW} are very large and are carried out using the computer algebra package Macaulay2.

In this paper we give a new proof of Theorem \ref{thm:coinvlef} that is both conceptual, in that it does not involve heavy computations, and elementary, in that it uses only algebra, combinatorics of finite reflection groups and some basic Schubert calculus.

One of the key ideas in this paper stems from algebraic topology.  Let $G$ be a semi-simple linear algebraic group and $B\subseteq G$ a Borel subgroup.  If $P\supseteq B$ is a parabolic subgroup of $G$ there is a fiber bundle of topological spaces
\[\xymatrix{P/B\ar[r]^-{\iota} & G/B\ar[d]^-{\pi}\\
& G/P.\\}\]
The Leray-Hirsch theorem from algebraic topology (see \cite{Ha} page 432) implies that the cohomology ring of $G/B$ is a free module over the cohomology ring of $G/P$ with  basis in one-to-one correspondence with an $\R$-basis for the cohomology ring of $P/B$.  More precisely the induced map 
$$\pi^*\colon H(G/P)\rightarrow H(G/B)$$ 
gives $H(G/B)$ the structure of a free $H(G/P)$-module, and the induced map $$\iota^*\colon H(G/B)\rightarrow H(P/B)$$ 
is surjective with $\ker(\iota^*)=\left(H(G/P)\right)^+\cdot H(G/B)$.  Equivalently, for any graded vector space section $s\colon H(P/B)\rightarrow H(G/B)$ of $\iota^*$, the map
$$\xymap{H(G/P)\otimes_\R H(P/B)\ar[r]^-{\cong} & H(G/B)\\
b\otimes f\ar@{|->}[r] & \pi^*(b)\cdot s(f)\\}$$
is an $H(G/P)$-module isomorphism.

The cohomology ring of $G/P$ is isomorphic (via a degree halving map) to the ring of ``relative" coinvariants $S_W^{W'}$, where $W'\subset W$ is the parabolic subgroup corresponding to $P$.  The cohomology ring of $P/B$ is isomorphic to the coinvariant ring of $W'$.  Hence in terms of coinvariant rings, the Leray-Hirsch theorem implies that we have an isomorphism of $S^{W'}_W$-modules:
\begin{equation}
\label{eq:lhthm}
S^{W'}_W\otimes_\R S_{W'}\cong S_W
\end{equation}
It turns out that \eqref{eq:lhthm} holds for all pairs of finite reflection groups and parabolic subgroups, crystallographic or not.  We show that if the factors $S_W^{W'}$ and $S_{W'}$ have the strong Lefschetz property, then $S_W$ also has the strong Lefschetz property.  In fact we prove the following more general result phrased in terms of Lefschetz algebras:
\begin{theorem}
\label{thm:fibun}
Let $(B,\lambda)$ and $(F,\tau)$ be Lefschetz algebras.  Let $E=\bigoplus_{i=0}^eE^i$ be an $\N$-graded Artinian $\R$-algebra equipped with $\R$-algebra homomorphisms
$$\pi\colon B\rightarrow E$$
and
$$\iota\colon E\rightarrow F.$$
Suppose that 
\begin{enumerate}
\item $E$ is a \emph{free} $B$-module via $\pi$
\item $\iota$ is surjective with $\ker\{\iota\}=B^+\cdot E$.
\end{enumerate}
Then for any $x\in\iota^{-1}(\tau)$ the pair $(E,\pi(\lambda)+tx)$ is a Lefschetz algebra for some $t\in\R$.
\end{theorem}
The proof of Theorem \ref{thm:coinvlef} is by induction on the rank of $W$.  For each type of finite reflection group $W$ we give a (maximal) parabolic subgroup $W'\subset W$ and write $S_W\cong S_{W}^{W'}\otimes_\R S_{W'}$ as in \eqref{eq:lhthm}.  The Lefschetz computation for $S_W^{W'}$ is then reduced to a simple counting argument using Schubert calculus together with some other combinatorial tidbits, while the induction hypothesis implies that $S_{W'}$ has the strong Lefschetz property.   Theorem \ref{thm:fibun} then implies that $S_W$ also has the strong Lefschetz property.   

This paper is organized as follows.  In Section 2 we define Lefschetz algebras, describe some of their basic properties, and proceed to prove Theorem \ref{thm:fibun}.  In Section 3 we give some basic notions and results on finite reflection groups, coinvariant rings and their Schubert calculi.  We also establish a Leray-Hirsch type result as in \eqref{eq:lhthm}.  In Section 4 we investigate the Lefschetz properties of relative coinvariant rings using the classification of finite reflection groups to carry out a type-by-type analysis.  In Section 5 we combine the results obtained in the preceding sections to prove Theorem \ref{thm:coinvlef}.  In Section 6 we give some concluding remarks.

\section{Lefschetz Algebras and the Proof of Theorem \ref{thm:fibun}}
Throughout this paper all homomorphisms are graded of degree zero unless otherwise indicated.  We use the notation $R[i]$ to denote the graded object $R$ shifted up by $i$ (i.e. $(R[i])^j=R^{j+i}$).

An $\N$-graded Artinian $\R$-algebra is a commutative graded finite-dimensional ring of the form $R=\bigoplus_{i=0}^rR^i$ such that $R^0=\R$.  A \emph{Lefschetz element} is an element $l\in R^1$ such that the multiplication maps $l^{r-2i}\colon R^i\rightarrow R^{r-i}$ are isomorphisms for $0\leq i\leq\lfloor\frac{r}{2}\rfloor$.
 
\begin{definition}
\label{def:lefalg}
A \emph{Lefschetz algebra} is a pair $(R,l)$ consisting of an $\N$-graded Artinian $\R$-algebra $R$ together with a fixed Lefschetz element $l\in R^1$ for $R$.
\end{definition}

It will be convenient to think of the $\R$-algebra $R$ as a module over the polynomial ring in one variable $\R[X]$ (with the usual grading), where $X$ acts on $R$ by multiplication by $l$.  In fact any degree one endomorphism $A\colon R\rightarrow R[1]$ defines a graded $\R[X]$-module structure on $R$ by defining 
$$X^i\cdot f\coloneqq A^i(f).$$

The simplest non-trivial example of a Lefschetz algebra is the polynomial ring in one variable divided by a monomial:
\begin{equation}
\label{eq:simple}
P(n)\coloneqq\frac{\R[X]}{\langle X^{n+1}\rangle}.
\end{equation}
Here the fixed Lefschetz element is the equivalence class of $X\in\R[X]$.
Given a Lefschetz algebra $(R,l)$, where $R=\bigoplus_{i=0}^r R^i$, define the homogeneous subspace $P=\bigoplus_{i=0}^rP^i\subset R$ by
$$P^i\coloneqq
\begin{cases}
\ker\left\{l^{r-2i+1}\colon R^i\rightarrow R^{r-i+1}\right\} & \text{if $0\leq i\leq\lfloor\frac{r}{2}\rfloor$}\\
P^i=0 & \text{if $i>\frac{r}{2}$}.\\
\end{cases}$$
The subspace $P$ is called the \emph{primitive} subspace of $(R,l)$.

\begin{primdecomp*}
\label{prop:primdecomp}
There is an isomorphism of vector spaces
$$R=\bigoplus_{i=0}^{\lfloor\frac{r}{2}\rfloor}\left(\bigoplus_{j=0}^{r-2i}l^j\cdot P^i\right).$$
\end{primdecomp*}
\begin{proof}
The proof is easy and we leave it as an exercise for the reader.
\end{proof}

The following lemma plays an important role in the proof of Theorem~\ref{thm:fibun}.
\begin{lemma}
\label{lem:prod}
Let $(U,\mu)$ and $(V,\nu)$ be two Lefschetz algebras.  Define the $\N$-graded Artinian $\R$-algebra 
$$W\coloneqq U\otimes_\R V.$$
Let $\omega\coloneqq\mu\otimes 1+1\otimes\nu\in W^1$.  Then $(W,\omega)$ is a Lefschetz algebra.
\end{lemma}
\begin{proof}
We only sketch the proof here and leave the details to the reader.  One can show, using the Primitive Decomposition Theorem, that it suffices to prove Lemma \ref{lem:prod} for the case where $U=P(n)$ and $V=P(m)$, for $n\leq m$.  Taking $X$, $Y$ to be the Lefschetz elements for $U$, $V$ respectively, our candidate Lefschetz element for $W$ is $\omega\coloneqq X+Y\in W=\R[X,Y]\big/\langle X^n,Y^m\rangle$.  It is straightforward to compute that the matrix for the Lefschetz map
$$\omega\colon W^i\rightarrow W^{d-i}$$ in the natural monomial basis (ordered by powers of $X$) is given by $\left(C^i_{jk}\right)$ where 
\begin{equation}
\label{eq:binomialtensor}
C^i_{jk}=\begin{cases}
\binom{d-2i}{n-i+j-k} \ \ \ 0\leq j,k\leq i & \text{if} \ \ \ 0\leq i\leq n\leq m\leq d\\
\binom{d-2i}{j-k} \ \ \ 0\leq j,k\leq n & \text{if} \ \ \ 0\leq n\leq i\leq m\leq d.\\
\end{cases}
\end{equation}
By a theorem in \cite{AZ} Chapter 23, the determinant of the matrix with binomial entries \eqref{eq:binomialtensor} is non-zero.  Barthel, Brasselet, Fieseler and Kaup give an alternative coordinate free proof in \cite{BBFK2} Proposition 5.7.
\end{proof}

The remainder of this section will be devoted to the proof of Theorem \ref{thm:fibun}.
With notations as in the statement of Theorem \ref{thm:fibun}, let $B=\bigoplus_{i=0}^bB^i$ and $F=\bigoplus_{i=0}^f F^i$ and fix an element $x\in\iota^{-1}(\tau)\in E^1$.
Consider $E$ as a $B$-module via $\pi$ and let $\End^1_B(E)$ denote the graded $B$-module endomorphisms of degree $1$.  Any choice of $A\in\End_B^1(E)$ endows 
$E$ with a $B[X]$-module structure by the prescription
$$(bX^i)\cdot e\coloneqq \pi(b)\cdot(A^i(e))$$
for all $b\in B$ and $e\in E$, where the dot on the RHS denotes multiplication in $E$.

Note that any $\R$-vector space section 
$$s\colon F\rightarrow E$$
of the surjective ring homomorphism $\iota\colon E\rightarrow F$ yields a $B$-module isomorphism
\[\xymatrixcolsep{5pc}\xymatrixrowsep{1pc}\xymatrix{B\otimes_\R F\ar[r]^-{\pi\otimes s} & E\\
b\otimes f\ar@{|->}[r] & \pi(b)\cdot s(f)\\}\]
by conditions (i) and (ii) in the statement of Theorem \ref{thm:fibun}.

We fix a section $s$ as follows.  First choose and fix \emph{any} vector space section 
$$\tilde{s}\colon F\rightarrow E$$
and define the homogeneous subspace
$$\tilde{P}\coloneqq\tilde{s}(P)\subset E$$
where $P\subset F$ is the primitive subspace of the Lefschetz algebra $(F,\tau)$.

Define the homogeneous subspace 
$$\tilde{F}\coloneqq\bigoplus_{i=0}^{\lfloor\frac{f}{2}\rfloor} \left(\bigoplus_{j=0}^{f-2i}x^j\cdot\tilde{P}^i\right)\subset E.$$
Define a new vector space section with image $\tilde{F}\subset E$ by
\[\xymatrixcolsep{5pc}\xymatrixrowsep{1pc}\xymatrix{F\ar[r]^-s & E\\
\tau^j\cdot p\ar@{|->}[r] & x^j\cdot\tilde{s}(p)\\}\]
where $p\in P^i$ and $\tau^j$ (resp. $x^j$) denotes the element $\tau$ (resp. $x$) raised to the $j^{th}$ power for $0\leq j\leq f-2i$.  Thus we have fixed a $B$-module isomorphism
\[\xymatrixcolsep{5pc}\xymatrixrowsep{1pc}\xymatrix{B\otimes_\R F\ar[r]^-{\pi\otimes s} & E\\
b\otimes f\ar@{|->}[r] & \pi(b)\cdot s(f).\\}\]
The tensor product comes with a ``preferred" $B[X]$-module structure coming from the natural ring structure on $B\otimes_\R F$; that is 
$$(b'X^i)\cdot (b\otimes f)\coloneqq (b'\otimes\tau^i)\cdot(b\otimes f)=(b'\cdot b)\otimes (\tau^i\cdot f).$$
Recall that Lemma \ref{lem:prod} implies that the pair
$$(B\otimes_\R F,\lambda\otimes 1+1\otimes\tau)$$ 
is a Lefschetz algebra, hence if $\pi\otimes s$ were a $B[X]$-module homomorphism we would be done.
Note that the $B$-module homomorphism $\pi\otimes s$ is \emph{almost} a $B[X]$-module homomorphism, in the the following sense:  For $0\leq i\leq\lfloor\frac{f}{2}\rfloor$, $p\in P^i$, and $b\in B$ we have 
$$X\cdot\pi\otimes s\left(b\otimes \tau^j\cdot p\right)
=\pi\otimes s\left(X\cdot\left(b\otimes \tau^j\cdot p\right)\right)$$
for $j<f-2i$, and for $j=f-2i$ we have 
\begin{align}
\label{eq:chern}
X\cdot\pi\otimes s\left(b\otimes \tau^{f-2i}\cdot p\right) & =\pi(b)\cdot\left(x^{f-2i+1}\cdot s(p)\right)\nonumber\\
& =\pi(b)\cdot\left(\sum_{\ell,m}c_{\ell m}x^\ell\cdot s(p'_{m})\right)
\end{align}

for some $c_{\ell m}\in \pi(B^+)$.

The idea is to define a one-parameter family $A_t\in\End_B^1(E)$ that will ``continuously deform" the $B[X]$-module structure on $E$ from its given structure (where multiplication by $X$ is multiplication by $x\in E^1$) into one for which $\pi\otimes s$ \emph{is} a $B[X]$-module isomorphism.  Essentially this amounts to deforming the ``Chern classes", $\left\{c_{\ell m}\right\}$, in \eqref{eq:chern} to zero.  

For the remainder of this section, unless otherwise indicated, we let $p$ denote an arbitrary element in $P^i$ for appropriate $0\leq i\leq\lfloor\frac{f}{2}\rfloor$ and let $v$ denote the corresponding element $s(p)$ in $E^i$.
We will use the letters $b$ and $e$ to denote arbitrary homogeneous elements of $B$ and $E$ respectively, and will let $\tau^j$ (resp.\ $x^j$) denote the element $\tau$ (resp.\ $x$) raised to the $j^{th}$ power in $F$ (resp.\ $E$) for $0\leq j\leq f-2i$.    

For each $t\in\R$ define the ring homomorphism
\[\xymatrixcolsep{5pc}\xymatrixrowsep{1pc}\xymatrix{B\ar[r]^-{\hat{\phi}_t} & B\\
b\ar@{|->}[r] & t^{\deg(b)}\cdot b.\\}\]
Note that $\hat{\phi}_t$ is a ring \emph{isomorphism} for $t\neq 0$ and  $$(\hat{\phi}_t)^{-1}=\hat{\phi}_{\frac{1}{t}}.$$
For each $t\in\R$, $\hat{\phi}_t$ extends to a $B$-module homomorphism
\[\xymatrixcolsep{5pc}\xymatrixrowsep{1pc}\xymatrix{E\ar[r]^-{\phi_t} & E\\}\]
that is ``twisted" in the sense that $\phi_t(\pi(b)\cdot e)=\pi(\hat{\phi}_t(b))\cdot\phi_t(e)$ for all $b\in B$ and $e\in E$.

Define the vector space maps
$$\hat{A}_{t,i}\colon\bigoplus_{j=0}^{f-2i}x^j\cdot\tilde{P}^i\rightarrow E[1]$$
by the formula
$$\hat{A}_{t,i}(x^j\cdot v)=
\phi_t(x^{j+1}\cdot v).$$
Note that $\hat{A}_{t,i}(x^j\cdot v)=x^{j+1}\cdot v$ for $j<f-2i$.
For each $t\in\R$ this defines a vector space map
$$\hat{A}_t\coloneqq\bigoplus_{i=0}^{\lfloor\frac{f}{2}\rfloor}\hat{A}_{t,i}\colon \tilde{F}\rightarrow E[1].$$
Since an $\R$-basis for $\tilde{F}$ is a $B$-module basis for $E$, these maps $\hat{A}_t$ extend $B$-linearly to $E$ to define a one-parameter family 
$$A_t\colon E\rightarrow E[1]$$
of $B$-module endomorphisms as desired.  

Next define for each $t\in\R$, $\R$-vector space maps
\[\xymatrixcolsep{5pc}\xymatrixrowsep{1pc}\xymatrix{\tilde{F}\ar[r]^-{\hat{\chiup}_t} & \tilde{F}\\
x^j\cdot v\ar@{|->}[r] & t^{j+i}x^j\cdot v.\\}\]
These maps extend uniquely to $B$-module endomorphisms \[\xymatrixcolsep{5pc}\xymatrixrowsep{1pc}\xymatrix{E\ar[r]^-{\chiup^{}_t} & E\\
\pi(b)\cdot(x^j\cdot v)\ar@{|->}[r] & \pi(b)\cdot(t^{j+i}x^j\cdot v).\\}\]
The morphisms $\chiup^{}_t$ and $\phi_t$ are related by the following composition law
\begin{equation}
\label{eq:composition}
\chiup^{}_t\circ\phi_t(e)=t^{\deg(e)}e \ \ \forall t\in\R, e\in E.
\end{equation}
We have the following important observation.

\begin{lemma}
\label{lem:chits}
The following diagram commutes:
\begin{equation}
\label{eq:chits}
\xymatrixcolsep{5pc}\xymatrixrowsep{4pc}\xymatrix{E\ar[d]_-{tx} & E\ar[l]_-{\chiup_t}\ar[d]^-{A_t}\\
E[1] & E[1]\ar[l]^-{\chiup_t}\\}
\end{equation}
where the left vertical map is multiplication by the element $tx\in E^1$.
\end{lemma}
\begin{proof} 
We need to show that $\chiup^{}_t\circ A_t=(tx)\cdot\chiup^{}_t$ and it suffices to check this on (homogeneous) elements of $E$ of the form $\pi(b)\cdot (x^j\cdot v)\coloneqq b\cdot(x^j\cdot v)$ (we omit the $\pi$ in this proof for notational convenience).  For $t=0$ we compute 
$$\chiup^{}_0\circ A_0(b\cdot(x^j\cdot v))=
b\cdot\chiup^{}_0\circ\phi_0(x^{j+1}\cdot v)$$ 
which is clearly zero in light of \eqref{eq:composition}.  For $t\neq 0$ the $B$-module homomorphisms $\chiup^{}_t$ are actually $B$-module \emph{isomorphisms} with 
$$(\chiup^{}_t)^{-1}=\chiup^{}_{\frac{1}{t}}.$$
Hence we compute 
\begin{align*}\chiup^{}_t\circ A_t\circ\chiup^{}_{\frac{1}{t}}(b\cdot (x^j\cdot v))= & \chiup^{}_t\circ A_t(\frac{1}{t^{j+i}}b\cdot (x^j\cdot v))\\
= & 
\chiup^{}_t\left(t^{-j-i}b\cdot\phi_t(x^{j+1}\cdot v)\right)\\
= & 
t^{-j-i}b\cdot\chiup^{}_t\circ\phi_t(x^{j+1}\cdot v)\\
= & 
tb\cdot(x^{j+1}\cdot v)\\
= & tx\cdot(b\cdot(x^j\cdot v));\\
\end{align*}
the second to last equality follows from \eqref{eq:composition}.  Hence 
\begin{equation}
\label{eq:cob}
\chiup_t\circ A_t\circ\chiup_{\frac{1}{t}}=tx
\end{equation}
and the diagram commutes for every $t\in\R$.  
\end{proof}
Equation \eqref{eq:cob} can be interpreted as a $B$-module change-of-base formula for the linear transformation $A_t$, $t\neq 0$.

The following lemma relates the $B[X]$-module structure on $E$ given by $A_t$ at $t=0$ with the preferred $B[X]$-module structure on the tensor product $B\otimes_\R F$.
\begin{lemma}
\label{lem:chio}
The following diagram commutes:
\begin{equation}
\label{eq:chios}
\xymatrixcolsep{5pc}\xymatrixrowsep{4pc}\xymatrix{(B\otimes_\R F)\ar[r]^-{\pi\otimes s}\ar[d]_-{1\otimes\tau} & E\ar[d]^-{A_0}\\
(B\otimes F)[1]\ar[r]_-{\pi\otimes s} & E[1].\\}
\end{equation}
\end{lemma}
\begin{proof}
As before, it suffices to check this for simple tensors of the form
$b\otimes\tau^j\cdot p$. 
We compute
\begin{equation}
\label{eq:comat0}
A_0\circ(\pi\otimes s)(b\otimes\tau^j\cdot p)=\begin{cases}
\pi(b)\cdot x^{j+1}\cdot s(p) & \text{if $j<f-2i$}\\
\pi(b)\cdot\phi_0(x^{f-2i+1}\cdot s(p)) & \text{if $j=f-2i$}.\\
\end{cases}
\end{equation}
Recall that $\iota(x^{f-2i+1}\cdot s(p))=\tau^{f-2i+1}\cdot p=0$, hence by assumption (ii), $x^{f-2i+1}s(p)\in B^+\cdot E$.  On the other hand we have that $\phi_0(B^+\cdot E)=0$.  Applying this observation to \eqref{eq:comat0} we get
$$A_0\circ(\pi\otimes s)(b\otimes(\tau^j\cdot p))=\begin{cases}
\pi(b)\cdot x^{j+1}\cdot s(p) & \text{if $j<f-2i$}\\
0 & \text{if $j=f-2i$}.\\
\end{cases}$$
On the other hand we compute
$$(\pi\otimes s)\circ(1\otimes\tau)(b\otimes(\tau^j\cdot p))
=\begin{cases}
\pi(b)\cdot x^{j+1}s(p) & \text{if $j<f-2i$}\\
0 & \text{if $j=f-2i$}.\\
\end{cases}$$
Hence the diagram commutes and this completes the proof of Lemma \ref{lem:chio}.
\end{proof}

We are now in a position to prove Theorem \ref{thm:fibun}.

\begin{proof}[Proof of Theorem \ref{thm:fibun}]
Consider the $\R$-vector space map
$$\Lambda+A_t\in\End_\R^1(E)$$
where 
$$\Lambda\colon E\rightarrow E[1]$$ is the map ``multiplication by $\pi(\lambda)$".  Note that 
\begin{equation}
\label{eq:chiupts}
\chiup^{}_t\circ\Lambda=\Lambda\circ\chiup^{}_t
\end{equation} 
for all $t\in\R$.  Fix $t\in\R$ and for each $0\leq k\leq\lfloor\frac{e}{2}\rfloor$, consider the map
\begin{equation}
\label{eq:lamA}
(\Lambda+A_t)^{e-2k}\colon E^k\rightarrow E^{e-k}.
\end{equation}
By the commutivity of \eqref{eq:chits} together with \eqref{eq:chiupts}, the following diagram also commutes:
\[\xymatrixcolsep{5pc}\xymatrixrowsep{4pc}\xymatrix{E\ar[d]_-{\pi(\lambda)+tx} & E\ar[l]_-{\chiup^{}_t}\ar[d]^-{\Lambda+A_t}\\
E[1] & E[1]\ar[l]^-{\chiup^{}_t}.\\}\]
Moreover for $t\neq 0$ the map $\chiup_t$ is an isomorphism.  Hence in order to show that the pair $(E,\pi(\lambda)+tx)$ is a Lefschetz algebra it suffices to show that for some $t\neq 0$ the map \eqref{eq:lamA} is an isomorphism for each $0\leq k\leq\lfloor\frac{e}{2}\rfloor$.

By the commutativity of \eqref{eq:chios}, the following diagram also commutes:
\[\xymatrixcolsep{5pc}\xymatrixrowsep{4pc}\xymatrix{(B\otimes_\R F)\ar[r]^-{\pi\otimes s}\ar[d]_-{\lambda\otimes 1+1\otimes\tau} & E\ar[d]^-{\Lambda+A_0}\\
(B\otimes F)[1]\ar[r]_-{\pi\otimes s} & E[1].\\}\]
Since  
$$((B\otimes_\R F),\lambda\otimes 1+1\otimes\tau)$$ 
is a Lefschetz algebra (by Lemma \ref{lem:prod}), we deduce that for each $0\leq k\leq\lfloor\frac{e}{2}\rfloor$ the map \eqref{eq:lamA} is an isomorphism for $t=0$.

Consider for each $0\leq k\leq\lfloor\frac{e}{2}\rfloor$, the function of the real variable $t$, $$D_k(t)\coloneqq\det\left(\left(\Lambda+A_t\right)^{e-2k}\right).$$
It is straightforward to see that $D_k(t)$ is a polynomial function of $t$.  Furthermore, since $D_k(0)\neq 0$, the polynomial $D_k(t)$ is not identically zero for all $0\leq k\leq\lfloor\frac{e}{2}\rfloor$.  Hence there must be some value $0\neq t_0\in\R$ such that $D_k(t_0)\neq 0$ for all $0\leq k\leq\lfloor\frac{e}{2}\rfloor$.  Therefore $(E,\pi(\lambda)+t^{}_0x)$ is a Lefschetz algebra, and this completes the proof of Theorem \ref{thm:fibun}.
\end{proof} 

\section{Finite Reflection Groups and Coinvariant Rings}
In this section we give some basic results on finite reflection groups and their coinvariant rings, following  \cite{BGG}, \cite{BB}, \cite{Hil}, \cite{Hi} and \cite{Hum}.

\subsection{Finite reflection groups}
Fix an inner product $\langle,\rangle$ on $\R^n$ and a (reduced) root system $\Phi\subset\R^n$.  For each $\gamma\in\Phi$ let $\check{\gamma}\in(\R^n)^*$ denote the corresponding co-root, defined in terms of the inner product $\langle,\rangle$ by 
\begin{equation}
\label{eq:coroot}
\check{\gamma}(x)=2\frac{\langle x,\gamma\rangle}{\langle\gamma,\gamma\rangle}.
\end{equation}  
The reflection corresponding to $\gamma\in\Phi$ is the orthogonal transformation $s_\gamma\colon\R^n\rightarrow\R^n$ defined by $s_\gamma(x)=x-\check{\gamma}(x)\cdot\gamma$.  Let $W$ be the finite reflection group generated by the reflections $\left\{s_\gamma\mid \gamma\in\Phi\right\}$.  Fix a simple system $\Delta\subset\Phi$ and let $\Phi^+\subset\Phi$ be the corresponding positive system.  We let $\ell(w)$ denote the length of an element $w\in W$, and let $w_0\in W$ denote the longest element in $W$, with $\ell(w_0)=d$.  

A useful fact is that simple reflections permute a large subset of the positve roots:
\begin{equation}
\label{thm:simprefpos}
s_\alpha(\Phi^+\setminus\{\alpha\})=\Phi^+\setminus\{\alpha\}
\end{equation} 
for all $\alpha\in\Delta$; see \cite{Hum}, Proposition 1.4 for a proof of this fact.

For each $\alpha\in\Phi^+$ and each $w\in W$, either $\ell(s_\alpha\cdot w)>\ell(w)$ or $\ell(s_\alpha\cdot w)<\ell(w)$.  As in \cite{Hum} Proposition 5.7, we have:
\begin{equation}
\label{thm:posinclen}
\ell(s_\alpha\cdot w)>\ell(w)\Leftrightarrow w^{-1}(\alpha)\in\Phi^+.
\end{equation}

For elements $w',w\in W$ we will write $\xymatrixcolsep{1pc}\xymatrix{w'\ar[r]^\alpha & w}$ to mean that $\alpha\in\Phi^+$, $s_\alpha\cdot w'=w$ and $\ell(s_\alpha\cdot w')=\ell(w')+1$.  There is a natural partial order on the set $W$ called the \emph{Bruhat ordering}, defined as follows:  Set $w\leq w'$ if and only if there exist group elements $w_1,\ldots,w_N$ and positive roots $\beta_0,\ldots,\beta_N$ such that 
$$\xyto{w\ar[r]^-{\beta_0} & w_1\ar[r]^-{\beta_1} & \cdots w_N\ar[r]^-{\beta_N} & w'}.$$
See \cite{Deo} or \cite{BB} Chapter 2 for more information and references on the Bruhat ordering of a finite reflection group.

\subsection{Coinvariant rings}
Let $S=\Sym(\R^n)$ be the polynomial ring on $(\R^n)^*$.  Let $S^W\subseteq S$ denote the (graded) sub-ring of invariant polynomials under $W$ and let $(S^W)^+\subseteq S^W$ denote the ideal generated by the invariants of positive degree.  Let $I\coloneqq(S^W)^+\cdot S\subseteq S$ denote the corresponding ideal in $S$.  Let $S_W\coloneqq S/I$ be the coinvariant ring.  In \cite{BGG}, Bernstein, Gelfand, and Gelfand introduced a set of operators on $S$ that are useful in studying the $S^W$-module structure on $S$.

\begin{definition}
For $\gamma\in\Phi^+$, define the operator
$A_\gamma\colon S\rightarrow S[-1]$ by the formula
$$A_\gamma(f)=\frac{f-s_\gamma(f)}{\gamma}$$
\end{definition}
Since for $x\in S^1=\R^n$ we have $s_\gamma(x)=x-\check{\gamma}(x)\gamma$, the quotient $\frac{x-s_\gamma(x)}{\gamma}$ is the real number $\check{\gamma}(x)$.  Since $S$ is generated in degree one, it follows that $A_\gamma$ is well-defined.  The operator $A_\gamma$ enjoys a Leibniz-type rule which the reader can readily verify:
\begin{equation}
\label{prop:multinvar2}
A_\gamma(f\cdot g)=f\cdot A_\gamma(g)+s_\gamma(g)\cdot A_\gamma(f)
\end{equation}
for all $f,g\in S$.  In particular, $A_\gamma(f)=0$ for all $f\in S^W$, hence $A_\gamma(I)\subseteq I$ for all $\gamma\in\Phi^+$.

Label the simple roots $\Delta=\{\gamma_1,\ldots,\gamma_k\}$ and let $s_k\in W$ denote the reflection corresponding to the simple root $\gamma_k$.  
Given an expression $w=s_{i_1}\cdots s_{i_r}$ (not necessarily reduced) define the operator
\begin{equation}
\label{eq:fundbgg}
A_{(i_1,\ldots,i_r)}\coloneqq A_{\gamma_{i_1}}\circ\cdots\circ A_{\gamma_{i_r}}\colon S\rightarrow S[-r].
\end{equation}
The following result of Bernstein, Gelfand and Gelfand \cite{BGG} is fundamental to Schubert calculus.
\begin{proposition}
\label{thm:fundbgg}
\begin{enumerate}
\item if $\ell(w)<r$ (i.e. the expression $w=s_{i_1}\cdots s_{i_r}$ is not reduced) then the operator $A_{(i_1,\ldots,i_r)}$ is zero.
\item if $\ell(w)=r$ (i.e. the expression $w=s_{i_1}\cdots s_{i_r}$ is reduced) then the operator $A_{(i_1,\ldots,i_r)}$ depends only on the element $w$; it is independent of the reduced expression for $w$.
\end{enumerate}
\end{proposition}
\begin{proof}
See \cite{BGG}, Theorem 3.4 or \cite{Hi}, Proposition 2.6.
\end{proof}
Define the \emph{BGG-operator} for $w$, $A_w\colon S\rightarrow S[-r]$, to be the operator in \eqref{eq:fundbgg} with respect to any reduced expression; this is well defined by Proposition \ref{thm:fundbgg}.  

An immediate consequence of Proposition \ref{thm:fundbgg} is the following composition rule for BGG-operators:
\begin{equation}
\label{prop:multinvar1}
A_u\circ A_v=\begin{cases}
A_{u\cdot v} & \text{if $\ell(u\cdot v)=\ell(u)+\ell(v)$}\\
0 & \text{otherwise}\\
\end{cases}
\end{equation}
for all $u,v \in W$.

An element $\chi\in S^1$ defines an operator of degree $1$, $\chi\colon S\rightarrow S[1]$, given by multiplication by $\chi$ in $S$.  An element $w\in W$ also defines an operator of degree $0$, $w\colon S\rightarrow S$, given by the linear action of $W$ on $S$. Bernstein, Gelfand and Gelfand \cite{BGG} derive the following useful formula for the commutator of the operators $w^{-1}\circ A_w$ and $\chi$ that is crucial to the main results of this paper. 
\begin{proposition}
\label{thm:corrbgg}
For each $\chi\in S^1$ and each $w\in W$
\begin{equation}
\label{eq:corrbgg}
[w^{-1}\circ A_w,\chi]=\sum_{\scriptsize\xymatrixcolsep{1pc}\xymatrix{w'\ar[r]^{\alpha}& w}}\check{\alpha}(w'(\chi))\cdot w^{-1}\circ A_{w'}
\end{equation}
where the sum is taken over all $w'\in W$ and $\alpha\in\Phi^+$ such that $\xymatrixcolsep{1pc}\xymatrix{w'\ar[r]^\alpha & w}$.
\end{proposition}
\begin{proof}
See \cite{BGG}, Lemma 3.5 or \cite{Hi}, Theorem 4.1.
\end{proof}  

Equation \eqref{eq:corrbgg} is tailor-made for dealing with the Lefschetz problem as it allows us to compute ``integrals" of powers of degree $1$ elements in terms of weighted chains in the Bruhat order.
\begin{corollary}
\label{cor:corcorbgg}
Fix $w\in W$ of length $t$.  Let $\chi_1,\ldots,\chi_s\in S^1$ for some $s\leq t$ and let $\sigma\in S^{t-s}$.  Then 
\begin{equation}
\label{eq:sumoverpaths}
A_w(\chi_1\cdots\chi_s\cdot\sigma)=\sum_{\scriptsize{\xyto{u_1\ar[r]^-{\beta_1} & u_2 \cdots u_s\ar[r]^-{\beta_s} & w\\}}}\check{\beta}_1(u_1(\chi_1))\cdots\check{\beta}_s(u_s(\chi_s))\cdot A_{u_1}(\sigma).
\end{equation} 
\end{corollary}
\begin{proof}
The proof is by induction on $s\geq 0$, the base case being trivial.  Assume the result holds for any product of $s-1$ linear forms.  Then write 
\begin{align*} A_w(\chi_1\cdots\chi_s\cdot\sigma)= 
& A_w\circ\chi_s(\chi_1\cdots\chi_{s-1}\cdot\sigma)\\
= & (w^{-1}\circ A_w)\circ\chi_s(\chi_1\cdots\chi_{s-1}\cdot\sigma)\\
= & \chi_s\circ (w^{-1}\circ A_w)(\chi_1\cdots\chi_{s-1}\cdot\sigma)+ [w^{-1}\circ A_w,\chi_s](\chi_1\cdots\chi_{s-1}\cdot\sigma)\\
= & [w^{-1}\circ A_w,\chi_s](\chi_1\cdots\chi_{s-1}\cdot\sigma)\\
= & \sum_{\scriptsize\xymatrixcolsep{1pc}\xymatrix{u_s\ar[r]^{\beta_s}& w}}\check{\beta}_s(u_s(\chi_s))\cdot w^{-1}\circ A_{u_s}(\chi_1\cdots\chi_{s-1}\cdot\sigma)\\
= & \sum_{\scriptsize\xymatrixcolsep{1pc}\xymatrix{u_s\ar[r]^{\beta_s}& w}}\check{\beta}_s(u_s(\chi_s))\cdot A_{u_s}(\chi_1\cdots\chi_{s-1}\cdot\sigma)
\end{align*}
where the second to last equality follows from Proposition \ref{thm:corrbgg} and the last (as well as the second) equality follows from the fact that polynomials in degree $0$ are $W$-invariant.  Hence by induction, the assertion of Corollary \ref{cor:corcorbgg} holds.
\end{proof}

Viewing $S_W=\bigoplus_{i=0}^d\left(S_W\right)^i$ as a graded $\R$-vector space, define the graded vector space $T=\bigoplus_{i=0}^d T^i$ by $$T^i\coloneqq\Hom_\R\left(\left(S_W\right)^i,\R\right).$$
By \eqref{prop:multinvar2} the operator $A_w\colon S\rightarrow S[-\ell(w)]$ passes to an operator on the quotient 
\begin{equation}
\label{eq:Toper}
\bar{A}_w\colon\left(S_W\right)\rightarrow\left(S_W\right)[-\ell(w)].
\end{equation}  
In particular, if $\ell(w)=i$ the restriction of \eqref{eq:Toper} to $\left(S_W\right)^i$ is an element of $T^i$, which by abuse of notation we call by the same name.  A basic fact from Schubert calculus is that these elements $\bigsqcup_{i=0}^d\left\{\bar{A}_w\mid \ell(w)=i\right\}$ then form a homogeneous vector space basis for $T$.  The dual basis $\bigsqcup_{i=1}^d\left\{X_w\mid \ell(w)=i\right\}\subset S_W$ is called a \emph{Schubert basis} for the coinvariant ring $S_W$.  

\subsection{Parabolic subgroups and relative coinvariants}
Fix a subset $\Theta\subset\Delta$ of simple roots and let $\Phi_\Theta\subset\Phi$ denote the corresponding root system, with positive system $\Phi_\Theta^+$.  Let $W_\Theta\subset W$ denote the parabolic subgroup corresponding to $\Theta$.  $W_\Theta$ also acts on $\R^n$ and thus on $S$ by restricting the action of $W$.  Let $S^{W_\Theta}$ denote the invariant ring of $W_\Theta$.  Note that $S^W$ is naturally a sub-ring of $S^{W_\Theta}$.  Let $I_\Theta\subset S$ denote the ideal generated by the positive degree invariants of $W_\Theta$ and let $S_{W_\Theta}=S/I_\Theta$ denote the coinvariant ring of $W_\Theta$.  
Since $I\subseteq I_\Theta$, there is a natural surjection of rings
$$\iota\colon S_W\rightarrow S_{W_\Theta}$$ 
induced by the identity map on $S$.

The action of $W$ on $S$ induces an action on the quotient, $S_W$.  Let $S^{W_\Theta}_W$ denote the sub-ring (of $S_W$) of $W_\Theta$-invariants called the ring of \emph{relative coinvariants} (with respect to $W_\Theta\subset W$).  Let
$$\pi\colon S^{W_\Theta}_W\rightarrow S_W$$
denote the natural inclusion map.

\begin{lemma}
\label{lem:alternrelcoinv}
The diagram 
\[\xymatrixcolsep{5pc}\xymatrixrowsep{4pc}\xymatrix{S^{W_\Theta}\ar[r]^-{\text{incl}}\ar[d]_-{p|} & S\ar[d]^-{p}\ar[rd]^q & \\
S^{W_\Theta}_W\ar[r]_-{\text{incl}} & S_W\ar[r]_-{\iota} & S_{W_\Theta}\\}\]
commutes.  Moreover the left-most vertical map $p|$ is surjective and induces an isomorphism $S^{W_\Theta}_W\cong\frac{S^{W_\Theta}}{(S^W)^+S^{W_\Theta}}\cong S^{W_\Theta}\otimes_{S^W}\R$.
\end{lemma}
\begin{proof}
That the square in the diagram commutes follows from the observation that the quotient map $p\colon S\rightarrow S_W$ is $W-$equivariant.  The triangle commutes by the definition of $\iota$.  Thus the whole diagram must commute.

To see that $p|$ is surjective, take any $f\in\rel\subset S_W$ and let $F\in S$ be any lift.  Let $F^\sharp\in S^{W_\Theta}$ be the average of $F$ over $W_\Theta$.  Then $p|(F^\sharp)=f^\sharp=f$.
Note that $\ker(p|)=(S^{W})^+\cdot S\cap S^{W_\Theta}\supseteq(S^W)^+\cdot S^{W_\Theta}$.  The claim is that this containment is actually equality.  Indeed let $f\in\ker(p|)$; write $f=s_1g_1+\ldots+s_rg_r$ for some $s_j\in (S^W)^+$ and $g_j\in S$.  Averaging over $W_\Theta$ we get $f^\sharp=f=s_1g_1^\sharp+\ldots+s_rg_r^\sharp\in(S^W)^+\cdot S^{W_\Theta}$ which completes the proof.
\end{proof}

\begin{proposition}
\label{thm:coinvlh}
With $\pi\colon S_W^{W_\Theta}\rightarrow S_W$ and $\iota\colon S_W\rightarrow S_{W_\Theta}$ as above,  
\begin{enumerate}
\item $\pi$ makes $S_W$ a free $S_W^{W_\Theta}$-module of rank $|W_\Theta|=\dim_\R(S_{W_\Theta})$
\item $\iota$ is surjective with kernel $\left(S_W^{W_\Theta}\right)^+\cdot S_W$.
\end{enumerate}
\end{proposition}
\begin{proof}
To see (i), recall that $S$ is a free $S^{W_\Theta}$-module of rank $|W_\Theta|\coloneqq t$.  Therefore $S_W=S\otimes_{S^W}\R$ is a free $S^{W_\Theta}\otimes_{S^W}\R$-module of rank $t$.
By Lemma \ref{lem:alternrelcoinv}, $S^{W_\Theta}_W\cong S^{W_\Theta}\otimes_{S^W}\R$.  The assertion of (ii) follows immediately from the commutativity of the triangle in the diagram in Lemma \ref{lem:alternrelcoinv}.
\end{proof}

The set $W^\Theta\coloneqq\{w\in W\mid l(w\cdot s_\gamma)=l(w)+1 \ \forall \gamma\in\Theta\}$ is a complete list of distinct coset representatives of the quotient $W/W_\Theta$.  Moreover every element $w\in W$ can be expressed uniquely as a product
\begin{equation}
\label{lem:mincosrep}
w=\bar{w}\cdot\hat{w}
\end{equation}
where $\bar{w}\in W^\Theta$ and $\hat{w}\in W_\Theta$ and $\ell(w)=\ell(\bar{w})+\ell(\hat{w})$; see \cite{BB} Proposition 2.4.4 and Corollary 2.4.5 or 
\cite{Hi}, Theorem 5.1 and Corollary 5.2.  

\begin{lemma}
\label{lem:dimU}
We have $\dim_\R\left(\left(\rel\right)^i\right)=\#\left\{\bar{w}\in W^\Theta\mid \ell(\bar{w})=i\right\}$.
\end{lemma}
\begin{proof}
The proof is by induction on $i\geq 0$, the base case being trivial.  By Proposition \ref{thm:coinvlh}, we conclude that there is an isomorphism of graded $\R$-vector spaces
\begin{equation}
\label{eq:tensorcoinv}
S_W\cong\rel\otimes_\R S_{W_{\Theta}}.
\end{equation}
Using \eqref{eq:tensorcoinv} in conjunction with the induction hypothesis yields
\begin{align*}
\dim\left(\left(\rel\right)^i\right)= & \dim\left(\left(S_W\right)^i\right)- \sum_{j=0}^{i-1}\left(\dim\left(\left(\rel\right)^j\right)\right)\cdot \left(\dim\left(\left(S_{W_\Theta}\right)^{i-j}\right)\right)\\
= & \#\left\{\ell(w)=i\right\}-\sum_{j=0}^{i-1}\left(\#\left\{\ell(\bar{w})=j\right\}\right)\cdot \left(\#\left\{\ell(\hat{w})=i-j\right\}\right)\\
= & \#\left\{\ell(w)=i\right\}-\#\left\{w=\bar{w}\cdot\hat{w}\mid\ell(w)=i, \ 0\leq\ell(\bar{w})\leq (i-1)\right\}\\
= & \#\left\{w=\bar{w}\cdot\hat{w}\mid\ell(w)=\ell(\bar{w})=i\right\}\\
= & \#\left\{\bar{w}\in W^\Theta\mid\ell(\bar{w})=i\right\}\\
\end{align*}
as desired.
\end{proof}

\begin{proposition}
\label{prop:Ubasis}
The elements $\left\{X_{\bar{w}}\mid \bar{w}\in W^\Theta\right\}$ lie in $S_W^{W_\Theta}$.
\end{proposition}
\begin{proof}
See \cite{Hil} Corollary 4.2.
\end{proof}

Let $r\coloneqq\ell(\bar{w}_0)$.  Proposition \ref{prop:Ubasis} together with Lemma \ref{lem:dimU} implies that the elements $\bigsqcup_{i=1}^r\left\{X_{\bar{w}}\mid\bar{w}\in W_\Theta, \ \ \ell(\bar{w})=i\right\}$ are a vector space basis for the relative coinvariant ring $\rel$, called a \emph{relative Schubert basis} for $S_W^{W_\Theta}$.

\section{Lefschetz properties}
In this section we use tools from Section 3 to study the Lefschetz properties of the \emph{relative} coinvariant ring.

Define the vector $\rho\coloneqq\frac{1}{2}\sum_{\gamma\in\Phi^+}\gamma$.  Note that for $\alpha$ a simple root, $\check{\alpha}(\rho)=1$.  Indeed we can write $$s_\alpha(\rho)=s_\alpha(\rho-\frac{1}{2}\alpha)+\frac{1}{2}s_\alpha(\alpha).$$  Since $s_\alpha(\rho-\frac{1}{2}\alpha)=\rho-\frac{1}{2}\alpha$ (using \eqref{thm:simprefpos}) we see that $s_\alpha(\rho)=\rho-\alpha$.  On the other hand $s_\alpha(\rho)\coloneqq \rho-\check{\alpha}(\rho)\alpha$, hence $\check{\alpha}(\rho)=1$ as desired.  Consequently $\check{\alpha}(\rho)$ must be positive for every positive root $\alpha$.

Define $\rho_\Theta=\frac{1}{2}\sum_{\gamma\in\Phi^+_\Theta}\gamma$ and $\bar{\rho}=\frac{1}{2}\sum_{\gamma\in\Phi^+\setminus\Phi^+_\Theta}\gamma$; we have $\rho=\rho_\Theta+\bar{\rho}$.  Note that $\bar{\rho}\in S^{W_\Theta}$.  Indeed for $\alpha\in\Theta\subseteq\Delta$ \eqref{thm:simprefpos} implies that $$s_\alpha(\Phi^+\setminus\Phi^+_\Theta)\subseteq\Phi^+\setminus\Phi^+_\Theta.$$  Since $W_\Theta$ is generated by $s_\alpha$ ($\alpha\in\Theta$), we see that $W_\Theta$ just permutes the roots in $\Phi^+\setminus\Phi^+_\Theta$, hence preserves $\bar{\rho}$.  We have the following useful relationship between $\bar{\rho}$ and the roots in $\Phi^+\setminus\Phi^+_\Theta$.

\begin{lemma}
\label{sublem:relpos}
$\check{\alpha}(\bar{\rho})>0$ for all $\alpha\in\Phi^+\setminus\Phi^+_\Theta$.
\end{lemma}
\begin{proof}
First assume that $\alpha\in\Delta\setminus\Theta$.  Then $\check{\alpha}(\rho_\Theta)<0$ since $\langle\alpha,\alpha'\rangle<0$ for all $\alpha'\in\Delta\setminus\{\alpha\}$.  On the other hand we have already seen that $\check{\alpha}(\rho)>0$.  Hence we conclude that $\check{\alpha}(\bar{\rho})>0$ as well.

Now let $\alpha\in\Phi^+\setminus\Phi^+_\Theta$ be arbitrary.  There is a unique $\alpha_\Theta\in\spa_\R\{\Theta\}$ such that 
$$\alpha=\sum_{\gamma\in\Delta\setminus\Theta}c_\gamma\cdot\gamma+\alpha_\Theta$$
for some $c_\gamma\geq 0$.  Since $\bar{\rho}$ is $W_\Theta$-invariant, we must have $\langle\alpha_\Theta,\bar{\rho}\rangle=0$.  Thus we have 
\begin{equation}
\label{eq:relpositivity}
\langle\alpha,\bar{\rho}\rangle=\sum_{\gamma\in\Delta\setminus\Theta}c_\gamma \cdot\langle\gamma,\bar{\rho}\rangle > \ 0.
\end{equation}
Using formula \eqref{eq:coroot}, we conclude that $\check{\alpha}(\bar{\rho})>0$ for all $\alpha\in\Phi^+\setminus\Phi^+_\Theta$, as desired.
\end{proof}

By abuse of notation we use the symbol $\bar{\rho}$ to denote the equivalence class $\left[\bar{\rho}\right]$ in $S_W^{W_\Theta}$.  The element $\bar{\rho}\in\left(\rel\right)^1$ is our ``candidate" Lefschetz element in the relative coinvariant ring.  We want to show that the map 
\begin{equation}
\label{eq:lefmaps}
\bar{\rho}^{r-2i}\colon\left(\rel\right)^i\rightarrow\left(\rel\right)^{r-i}
\end{equation}
is an isomorphism for $0\leq i\leq\lfloor\frac{r}{2}\rfloor$.  
The idea is to use our relative Schubert basis to compute the matrix for \eqref{eq:lefmaps}.

Fix $0\leq i\leq\lfloor\frac{r}{2}\rfloor$ and let $\bar{u}\in W^\Theta$ be an element of length $i$.  Then we have
\begin{equation}
\label{eq:lefmat}
\bar{\rho}^{r-2i}\cdot X_{\bar{u}}=\sum_{\ell(\bar{v})=r-i}c^i_{\bar{v}\bar{u}}\cdot X_{\bar{v}}.
\end{equation}
Note that the coefficient $c^i_{\bar{v}\bar{u}}$ is just $A_{\bar{v}}(\bar{\rho}^{r-2i}\cdot X_{\bar{u}})$ for each $\ell(\bar{v})=r-i$ and $\ell(\bar{u})=i$.  We can compute these coefficients using Corollary \ref{cor:corcorbgg}.  

We write $\xyto{\bar{w}^\prime\ar[r]^\beta & \bar{w}}$ to mean that $\beta\in\Phi^+$, $\bar{w}',\bar{w}\in W^\Theta$, $s_\beta\cdot\bar{w}^\prime=\bar{w}$ and $\ell(s_\beta\cdot\bar{w}^\prime)=\ell(\bar{w}^\prime)+1$.

\begin{proposition}
\label{prop:matentr}
For each $\bar{u},\bar{v}\in W^\Theta$, with $\ell(\bar{u})=i$ and $\ell(\bar{v})=r-i$ we have
\begin{equation}
\label{eq:matentr}
A_{\bar{v}}(\bar{\rho}^{r-2i}\cdot X_{\bar{u}})=\sum_{\scriptsize{\xyto{\bar{u}=\bar{u}_1\ar[r]^-{\beta_1} & \bar{u}_2 \cdots \bar{u}_s\ar[r]^-{\beta_s} & \bar{v}}}}\check{\beta_1}(\bar{u}_1(\bar{\rho}))\cdots\check{\beta_s}(\bar{u}_s(\bar{\rho}))
\end{equation}
where $s=r-2i$.
\end{proposition}
\begin{proof}
Applying Corollary \ref{cor:corcorbgg} to the case where $\chi_i=\rho$, $1\leq i\leq s$, and $\sigma= X_{\bar{u}}$, we have
\begin{equation}
\label{eq:matentr1}
A_{\bar{v}}(\bar{\rho}^{r-2i}\cdot X_{\bar{u}})=\sum_{\scriptsize{\xyto{\bar{u}=u_1\ar[r]^-{\beta_1} & u_2 \cdots u_s\ar[r]^-{\beta_s} & \bar{v}\\}}}\check{\beta}_1(u_1(\bar{\rho}))\cdots\check{\beta}_s(u_s(\bar{\rho})).
\end{equation}
It remains to show that each non-zero term in \eqref{eq:matentr1} comes from a ``path" in $W^\Theta$.  Note that if $\xyto{u\ar[r]^-{\beta} & v\\}$ with $u\in W^\Theta$ and $v\notin W^\Theta$ then $\check{\beta}(u(\bar{\rho}))=0$.  Indeed if $v\notin W^\Theta$ then there is a reduced expression $v=s_1\cdots s_r$ with $\alpha_r\in \Theta$.  On the other hand $\xyto{u\ar[r]^-{\beta} & v\\}$ implies that $v\cdot s_{u^{-1}(\beta)}=u$ and $u\in W^\Theta$ implies $u=s_1\cdots s_{r-1}$.  Hence $u^{-1}(\beta)=\alpha_r\in\Phi_\Theta^+$ which forces $\check{\beta}(u(\rho))=0$.  We conclude that the only non-zero summands in \eqref{eq:matentr1} come from ``paths" in $W^\Theta$, as desired.
\end{proof}

\begin{remark*}
A result of Deodhar (see \cite{Deo}, Corollary 3.8) states that given two elements $\bar{w}_1, \bar{w}_2\in W^\Theta$ with $\bar{w}_1\leq\bar{w}_2$ (where $\leq$ denotes the Bruhat ordering on $W$), there exist elements $\bar{u}_1,\ldots, \bar{u}_r\in W^\Theta$ such that 
$$\xyto{\bar{w}_1=\bar{u}_0\ar[r]^-{\beta_0} & \bar{u}_1 \cdots \bar{u}_r \ar[r]^-{\beta_{r}} & \bar{u}_{r+1}=\bar{w}_2\\}.$$  
In particular this guarantees that the sum in \eqref{eq:matentr} is never vacuous.
\end{remark*}

We can express our matrix for the map in \eqref{eq:lefmaps} as
\begin{equation}
\label{eq:matrix}
\left(\begin{array}{ccc}
& \vdots &\\
\cdots& c^i_{\bar{v}\bar{u}} &\cdots \\
& \vdots & \\
\end{array}\right)_{\substack{\ell(\bar{v})=r-i\\ 
\ell(\bar{u})=i\\}}
\end{equation}  
where entry $c^i_{\bar{v}\bar{u}}$ is given by \ref{eq:matentr}.  Proposition \ref{prop:matentr} then gives a nice combinatorial interpretation to the matrix \eqref{eq:matrix} as a \emph{weighted path matrix} with respect to an appropriately weighted directed (acyclic) graph.  This interpretation lends itself to combinatorial tools effective in studying the determinant of \eqref{eq:matrix}.  

The classification of finite reflection groups says that a finite reflection group is either irreducible or is a direct product of irreducible ones (see \cite{Hum} Chapter 2).  Thus by Lemma \ref{lem:prod}, it suffices to prove Theorem \ref{thm:coinvlef} for $W$ irreducible.  Any given irreducible finite reflection group is of one of the ten \emph{types} shown in Table \ref{fig:class}.  For the remainder of this section we will use this list to exhibit, in each type of $W$,  a parabolic subgroup $W_\Theta\subset W$ for which the relative coinvariant ring $\rel$ has the strong Lefschetz property.  

\begin{center} 
\begin{table}[t]
\caption{classification of finite reflection groups}
\begin{center}
\begin{tabular}{|c|l|p{1in}|}
\hline
$A_n \ \ (n\geq 1)$ & $\xymatrixcolsep{3pc}\xymatrix{\circ\ar@{-}[r]&\circ&\cdots&\circ\ar@{-}[r]&\circ\\
&&&&\\}$\\
\hline
$B_n \ \ (n\geq 2)$ & $\xymatrixcolsep{3pc}\xymatrix{\circ\ar@{-}[r]&\circ&\cdots&\circ\ar@2{-}[r]&\circ\\
&&&&\\}$\\

\hline
$D_n \ \ (n\geq 4)$ & 
$\xymatrixcolsep{3pc}\xymatrix{&&&&&\circ\\
\circ\ar@{-}[r]&\circ&\cdots&\circ\ar@{-}[r]&\circ\ar@{-}[ur]\ar@{-}[dr] &\\
&&&&&\circ\\}$\\

\hline

$E_6$ & $\xymatrixcolsep{3pc}\xymatrix{&&\circ&&\\
\circ\ar@{-}[r]&\circ\ar@{-}[r]&\circ\ar@{-}[r]\ar@{-}[u]& \circ\ar@{-}[r]&\circ}$\\

\hline

$E_7$ & $\xymatrixcolsep{3pc}\xymatrix{&&\circ&&&\\
\circ\ar@{-}[r]&\circ\ar@{-}[r]&\circ\ar@{-}[r]\ar@{-}[u]& \circ\ar@{-}[r]&\circ\ar@{-}[r]&\circ}$\\

\hline

$E_8$ & $\xymatrixcolsep{3pc}\xymatrix{&&\circ&&&&\\
\circ\ar@{-}[r]&\circ\ar@{-}[r]&\circ\ar@{-}[r]\ar@{-}[u]& \circ\ar@{-}[r]&\circ\ar@{-}[r]&\circ\ar@{-}[r]&\circ}$\\

\hline

$F_4$ & $\xymatrixcolsep{3pc}\xymatrix{&&&\\
\circ\ar@{-}[r]&\circ\ar@{-}[r]& \circ\ar@2{-}[r]&\circ\\}$\\

\hline

$H_3$ & $\xymatrixcolsep{3pc}\xymatrix{&&\\
\circ\ar@3{-}[r]&\circ\ar@{-}[r]& \circ\\
}$\\

\hline

$H_4$ & $\xymatrixcolsep{3pc}\xymatrix{&&&\\
\circ\ar@3{-}[r]&\circ\ar@{-}[r]& \circ\ar@{-}[r]&\circ\\}$\\
\hline

$I_2(m)$ & $\xymatrixcolsep{3pc}\xymatrix{&\\
\circ\ar@{-}[r]^-m &\circ\\}$\\
\hline
\end{tabular}
\end{center}
\label{fig:class}
\end{table}
\end{center}

Let $G$ be the graph with vertex set $W^\Theta$ where two elements $\bar{w}'$ and $\bar{w}$ are joined by an edge if and only if $\xyto{\bar{w}'\ar[r]^-{\beta} & \bar{w}}$ for some $\beta\in\Phi^+$.  The graph $G$ is the Hasse diagram of the Bruhat order on $W$ restricted to $W^\Theta$ (by \cite{Deo} Corollary 3.8), hence it is naturally \emph{directed} and \emph{acyclic}.  To each (directed) edge of $G$, $e=\xyto{\bar{w}'\ar[r]^-{\beta} & \bar{w}}$, we assign the weight 
\begin{equation}
\label{eq:edgewts}
\omega(e)\coloneqq\check{\beta}(\bar{w}'(\bar{\rho})) \ > \ 0.
\end{equation}  
Note that the weights in \eqref{eq:edgewts} are positive.  Indeed $\xyto{\bar{u}\ar[r]^\beta & \bar{v}}$ implies $\bar{u}^{-1}(\beta)\in\Phi^+\setminus\Phi^+_\Theta$ by Proposition \ref{thm:posinclen}.  Now apply Lemma \ref{sublem:relpos}.

To each directed path $P$ in this graph, we define its weight by $$\omega(P)\coloneqq\prod_{e\in P}\omega(e).$$
Let $V^i$ denote the elements of $W^\Theta$ of length $i$, labelled $\{\bar{u}_1,\ldots,\bar{u}_N\}$ and label the elements $V^{r-i}$ as $\{\bar{v}_1,\ldots,\bar{v}_N\}$.  Then the $(j,k)^{th}$ entry $c^i_{\bar{v}_j\bar{u}_k}\coloneqq c^i_{jk}$ in \eqref{eq:matrix} is given by 
$$c^i_{jk}=\sum_{\scriptsize{P\colon\bar{u}_k\rightarrow\bar{v}_j}}\omega(P),$$ where the sum is over all directed paths $P$ from $\bar{u}_k$ to $\bar{v}_j$.  

A \emph{path system} $\mathcal{P}$ from $V^i$ to $V^{r-i}$ is a permutation $\sigma\in S_N$ together with a collection of paths $P_a\colon \bar{u}_a\rightarrow\bar{v}_{\sigma(a)}$ for $1\leq a\leq N$.  The \emph{sign} of a path system is the sign of the corresponding permutation: $\sgn(\mathcal{P})\coloneqq\sgn(\sigma)$.  We define the \emph{weight} of a path system to be the product of the weights of the corresponding paths: $\omega(\mathcal{P})\coloneqq\prod_{a=1}^N\omega(P_a)$.  It is an easy exercise to verify the identity 
\begin{equation}
\label{eq:det}
\det\left(\left(c^i_{jk}\right)\right)=\sum_{\mathcal{P}}\sgn(\mathcal{P})\cdot\omega(\mathcal{P}),
\end{equation}
where the sum is taken over \emph{all} path systems from $V^i$ to $V^{r-i}$.  

It will be useful to reference the following easy Lemma.
\begin{lemma}
\label{lem:easy}
If $\mathcal{O}$ is a collection of path systems from $V^i$ to $V^{r-i}$ such that $\sgn\colon\mathcal{O}\rightarrow\left\{+,-\right\}$ is constant, then \begin{equation}
\label{eq:det2}
\det\left(\left(c^i_{jk}\right)\right)=\sum_{\mathcal{P}\in\mathcal{O}}\sgn(\mathcal{P})\cdot\omega(\mathcal{P})= \pm\sum_{\mathcal{P}\in\mathcal{O}}\omega(\mathcal{P})
\end{equation}
is non-zero.
\end{lemma}
\begin{proof}
This follows from the positivity of \eqref{eq:edgewts}.
\end{proof} 
Using Lemma \ref{lem:easy} we can show that $\rel$ has the strong Lefschetz property (for an appropriate choice of $W_\Theta$) for most types of $W$.
 
\begin{proposition}
\label{prop:altypeslefschetz}
Let $W$ be a finite reflection group of type $A$, $B$, $D$, $I_2(m)$ or $H_3$.  Then there is a parabolic subgroup $W_\Theta\subset W$ such that $S_W^{W_\Theta}$ has the strong Lefschetz property.
\end{proposition}  
\begin{proof}
For $W$ of type $A_{n}$, $B_n$, or $D_n$, choose a parabolic subgroup $W_\Theta\subset W$ of type $A_{n-1}$, $B_{n-1}$, or $D_{n-1}$, respectively.  For type $I_2(m)$ any non-trivial parabolic subgroup will work.  In type $H_3$ choose a parabolic subgroup $W_\Theta\subset W$ of type $I_2(5)$.  In all these cases a straightforward computation of the (relative) Poincar\'e polynomial, using Corollary 4.5 in \cite{Hil} and Table 1 in \cite{Hum}, reveals that 
$$\dim_\R\left(\rel\right)^i=1, \ \ \ \ \ \ \text{for all} \ \ 0\leq i\leq \Big\lfloor\frac{r-1}{2}\Big\rfloor.$$
By Lemma \ref{lem:easy} (taking $\mathcal{O}$ to be the set of all path systems from $V^i$ to $V^{r-i}$) we deduce that \ref{eq:det} is non-zero, and thus that $\rel$ has the strong Lefschetz property.
\end{proof}

To show that $S_W^{W_\Theta}$ has the strong Lefschetz property for the remaining types takes a little more work.  Our main tool is the following result of Gessel and Viennot \cite{GV}; we state it here without proof.  For an excellent treatment and further references on this result, see \cite{AZ}, Chapter 23.
\begin{proposition}
\label{prop:GV}
Let $\mathcal{O}^i_{vd}$ denote the set of \emph{vertex disjoint} path systems from $V^i$ to $V^{r-i}$ i.e.
no two distinct paths $P_a$ and $P_b$ in $\mathcal{P}$ share a common vertex.  Then
\begin{equation}
\label{eq:GV}
\det\left(\left(c^i_{jk}\right)\right)= \sum_{\mathcal{P}\in\mathcal{O}^i_{vd}}\sgn(\mathcal{P})\cdot\omega(\mathcal{P}).
\end{equation}
\end{proposition}  
Proposition \ref{prop:GV} makes short work of our task in few more types.
\begin{proposition}
\label{prop:FEE}
If $W$ is a finite reflection group of type $F_4$, $E_6$ or $E_7$, then there is a parabolic subgroup $W_\Theta$ such that $S_W^{W_\Theta}$ has the strong Lefschetz property.
\end{proposition}
\begin{proof}
The proof is a type-by-type inspection.

In type $F_4$, choose a parabolic subgroup $W_\Theta\subset W$ of type $B_3$.  The graph for the quotient $F_4\big/B_3$ is shown in Table \ref{tab:F4E6E7}.
Note that there is only one path system from $V^i$ to $V^{15-i}$ for $0\leq i\leq 3$.  Hence by Lemma \ref{lem:easy}, \eqref{eq:GV} is non-zero for $0\leq i\leq 3$.  For $4\leq i\leq 7$, note that there are exactly two path systems in $\mathcal{O}^i_{vd}$ which differ only in their restriction to $V^7\sqcup V^8$.  Since the path systems in $\mathcal{O}^7_{vd}$ have distinct weights, as shown in Table \ref{tab:F4E6E7}, \eqref{eq:GV} must be non-zero for $4\leq i\leq 7$.    

In type $E_6$, choose a parabolic subgroup $W_\Theta\subset W$ of type $D_5$.  The graph for the quotient $E_6\big/D_5$ is shown in Table \ref{tab:F4E6E7} (this graph was borrowed from \cite{BB} page 44).  Note that the path systems in $\mathcal{O}^i_{vd}$ have the same sign for \emph{all} $0\leq i\leq 8$.  Thus Lemma \ref{lem:easy} implies that \eqref{eq:GV} is non-zero for all $0\leq i\leq 8$.
\begin{center}
\begin{table}[t]
\begin{tabular}{|c|c|c|c|c|p{1in}|}
\hline
$\xyPoset{
\\
\\
\\
\\
\\
\\
\\
\\
\\
\\
\\
\\
&&*=0{\bullet}\ar@{-}[d]&&\\
&&*=0{\bullet}\ar@{-}[d]&&\\
&&*=0{\bullet}\ar@{-}[d]&&\\
&&*=0{\bullet}\ar@{-}[ld]\ar@{-}[rd]&&\\
&*=0{\bullet}\ar@{-}[rd]&&*=0{\bullet}\ar@{-}[ld]\ar@{-}[rd]&\\ &&*=0{\bullet}\ar@{-}[rd]\ar@{-}[ld]&&*=0{\bullet}\ar@{-}[ld]\\
&*=0{\bullet}\ar@{-}[ld]\ar@{-}[rd]&&*=0{\bullet}\ar@{-}[ld]&\\
*=0{\bullet}\ar@{-}[rrd]|<<<{22}\ar@{-}[d]_{11}&&*=0{\bullet}\ar@{-}[d]^{11}\ar@{-}[lld]|<<<{11}&&\\
*=0{\bullet}\ar@{-}[rd]&&*=0{\bullet}\ar@{-}[ld]\ar@{-}[rd]&&\\
&*=0{\bullet}\ar@{-}[rd]&&*=0{\bullet}\ar@{-}[ld]\ar@{-}[rd]&\\
&&*=0{\bullet}\ar@{-}[ld]\ar@{-}[rd]&&*=0{\bullet}\ar@{-}[ld]\\
&*=0{\bullet}\ar@{-}[rd]&&*=0{\bullet}\ar@{-}[ld]&\\
&&*=0{\bullet}\ar@{-}[d]&&\\
&&*=0{\bullet}\ar@{-}[d]&&\\
&&*=0{\bullet}\ar@{-}[d]&&\\
&&*=0{\bullet}&&\\}$
&
{\tiny$\xymatrixrowsep{.23pc}\xymatrix{
27\\
26\\
25\\
24\\
23\\
22\\
21\\
20\\
19\\
18\\
17\\
16\\
15\\
14\\
13\\
12\\
11\\
10\\
9\\
8\\
7\\
6\\
5\\
4\\
3\\
2\\
1\\
0\\}$}
&
$\xyPoset{
\\
\\
\\
\\
\\
\\
\\
\\
\\
\\
\\
&&&*=0{\bullet}\ar@{-}[d]&&\\
&&&*=0{\bullet}\ar@{-}[d]&&\\
&&&*=0{\bullet}\ar@{-}[d]&&\\
&&&*=0{\bullet}\ar@{-}[ld]\ar@{-}[rd]&&\\
&&*=0{\bullet}\ar@{-}[rd]&&*=0{\bullet}\ar@{-}[ld]\ar@{-}[rd]&\\ &&&*=0{\bullet}\ar@{-}[rd]\ar@{-}[ld]&&*=0{\bullet}\ar@{-}[ld]\\
&&*=0{\bullet}\ar@{-}[ld]\ar@{-}[rd]&&*=0{\bullet}\ar@{-}[ld]&\\
&*=0{\bullet}\ar@{-}[rd]\ar@{-}[ld]&&*=0{\bullet}\ar@{-}[rd]\ar@{-}[ld]&&\\
*=0{\bullet}\ar@{-}[rd]&&*=0{\bullet}\ar@{-}[ld]\ar@{-}[rd]&&*=0{\bullet}\ar@{-}[ld]&\\
&*=0{\bullet}\ar@{-}[rd]&&*=0{\bullet}\ar@{-}[ld]\ar@{-}[rd]&&\\
&&*=0{\bullet}\ar@{-}[rd]&&*=0{\bullet}\ar@{-}[ld]\ar@{-}[rd]&\\
&&&*=0{\bullet}\ar@{-}[rd]\ar@{-}[ld]&&*=0{\bullet}\ar@{-}[ld]\\
&&*=0{\bullet}\ar@{-}[rd]&&*=0{\bullet}\ar@{-}[ld]&\\
&&&*=0{\bullet}\ar@{-}[d]&&\\
&&&*=0{\bullet}\ar@{-}[d]&&\\
&&&*=0{\bullet}\ar@{-}[d]&&\\
&&&*=0{\bullet}&&\\}$
&
{\tiny$\xymatrixrowsep{.23pc}\xymatrix{
27\\
26\\
25\\
24\\
23\\
22\\
21\\
20\\
19\\
18\\
17\\
16\\
15\\
14\\
13\\
12\\
11\\
10\\
9\\
8\\
7\\
6\\
5\\
4\\
3\\
2\\
1\\
0\\}$}
&
$\xyPoset{&&&&*=0{\bullet}\ar@{-}[d]&&&&\\
&&&&*=0{\bullet}\ar@{-}[d]&&&&\\
&&&&*=0{\bullet}\ar@{-}[d]&&&&\\
&&&&*=0{\bullet}\ar@{-}[d]&&&&\\
&&&&*=0{\bullet}\ar@{-}[ld]\ar@{-}[rd]&&&&\\
&&&*=0{\bullet}\ar@{-}[rd]\ar@{-}[ld]&&*=0{\bullet}\ar@{-}[ld]&&&\\ &&*=0{\bullet}\ar@{-}[rd]&&*=0{\bullet}\ar@{-}[ld]\ar@{-}[rd]&&&&\\
&&&*=0{\bullet}\ar@{-}[rd]&&*=0{\bullet}\ar@{-}[ld]\ar@{-}[rd]&&&\\
&&&&*=0{\bullet}\ar@{-}[ld]\ar@{-}[rd]&&*=0{\bullet}\ar@{-}[ld]\ar@{-}[rd]&&\\
&&&*=0{\bullet}\ar@{-}[rd]&&*=0{\bullet}\ar@{-}[rd]\ar@{-}[ld]&&*=0{\bullet}\ar@{-}[rd]\ar@{-}[ld]&\\
&&&&*=0{\bullet}\ar@{-}[rd]\ar@{-}[ld]&&*=0{\bullet}\ar@{-}[rd]\ar@{-}[ld]&&*=0{\bullet}\ar@{-}[ld]\\
&&&*=0{\bullet}\ar@{-}[rd]\ar@{-}[ld]&&*=0{\bullet}\ar@{-}[rd]\ar@{-}[ld]&&*=0{\bullet}\ar@{-}[ld]&\\
&&*=0{\bullet}\ar@{-}[rd]&&*=0{\bullet}\ar@{-}[rd]\ar@{-}[ld]\ar@{-}[d]&&*=0{\bullet}\ar@{-}[ld]&&\\
&&&*=0{\bullet}\ar@{-}[rd]\ar@{-}[d]&*=0{\bullet}\ar@{-}[rd]\ar@{-}[ld]&*=0{\bullet}\ar@{-}[ld]\ar@{-}[d]&&&\\
&&&*=0{\bullet}\ar@{-}[rd]\ar@{-}[ld]&*=0{\bullet}\ar@{-}[d]&*=0{\bullet}\ar@{-}[ld]\ar@{-}[rd]&&&\\
&&*=0{\bullet}\ar@{-}[rd]\ar@{-}[ld]&&*=0{\bullet}\ar@{-}[rd]\ar@{-}[ld]&&*=0{\bullet}\ar@{-}[ld]&&\\
&*=0{\bullet}\ar@{-}[rd]\ar@{-}[ld]&&*=0{\bullet}\ar@{-}[rd]\ar@{-}[ld]&&*=0{\bullet}\ar@{-}[ld]&&&\\
*=0{\bullet}\ar@{-}[rd]&&*=0{\bullet}\ar@{-}[rd]\ar@{-}[ld]&&*=0{\bullet}\ar@{-}[ld]\ar@{-}[rd]&&&&\\
&*=0{\bullet}\ar@{-}[rd]&&*=0{\bullet}\ar@{-}[rd]\ar@{-}[ld]&&*=0{\bullet}\ar@{-}[ld]&&\\
&&*=0{\bullet}\ar@{-}[rd]&&*=0{\bullet}\ar@{-}[rd]\ar@{-}[ld]&&&&\\
&&&*=0{\bullet}\ar@{-}[rd]&&*=0{\bullet}\ar@{-}[rd]\ar@{-}[ld]&&&\\
&&&&*=0{\bullet}\ar@{-}[ld]\ar@{-}[rd]&&*=0{\bullet}\ar@{-}[ld]&&\\
&&&*=0{\bullet}\ar@{-}[rd]&&*=0{\bullet}\ar@{-}[ld]&&&\\
&&&&*=0{\bullet}\ar@{-}[d]&&&&\\
&&&&*=0{\bullet}\ar@{-}[d]&&&&\\
&&&&*=0{\bullet}\ar@{-}[d]&&&&\\
&&&&*=0{\bullet}\ar@{-}[d]&&&&\\
&&&&*=0{\bullet}&&&&\\}$\\
\hline
$F_4\big/B_3$ & $i$ & $E_6\big/D_5$ & $i$ & $E_7\big/E_8$\\
\hline
\end{tabular}
\caption{Hasse diagrams of the quotients $W\big/W_\Theta$}
\label{tab:F4E6E7}
\end{table}
\end{center}

In type $E_7$, choose a parabolic subgroup $W_\Theta\subset W$ of type $E_6$.  The graph for the quotient $E_7\big/ E_6$ is shown in Table \ref{tab:F4E6E7} (this graph was borrowed from \cite{P}).  For $0\leq i\leq 4$ and $9\leq i\leq 13$, the path systems in $\mathcal{O}^i_{vd}$ have the same sign, hence Lemma \ref{lem:easy} implies \eqref{eq:GV} is non-zero for these $i$.  For $5\leq i\leq 8$ we can reduce the computations to a simple count.  Note that all of the edge weights of $G$ are equal.  Indeed using the computations of the root system $\Phi$ of type $E_7$ described in \cite{Hum} page 43, we compute that 
$$\bar{\rho}=9\cdot\left(e_8-e_7+2e_6\right)$$
where $\left\{e_1,\ldots,e_8\right\}$ are the standard basis vectors in $\R^8$.
Then it is straightforward to check that
$$\check{\alpha}(\bar{\rho})=18$$
for all $\alpha\in\Phi^+\setminus\Phi_\Theta^+$.  Hence we need only show that the sign function does not split $\mathcal{O}^i_{vd}$ in half for $5\leq i\leq 8$.  Note that the sign of a path system $\mathcal{P}\in\mathcal{O}^i_{vd}$ is completely determined by its \emph{middle leg} i.e., the restriction of $\mathcal{P}$ to the vertex set $V^{12}\sqcup V^{13}\sqcup V^{14}\sqcup V^{15}$. 

By inspection of Table \ref{tab:F4E6E7}, there are a total of nine possible middle legs for $\mathcal{P}$ four of which are shown in Table \ref{tab:E7}.  Note that the middle legs in the same column in Table \ref{tab:E7} must have the same sign, and those in distinct columns must have distinct signs.  Also note that any path system $\mathcal{P}\in\mathcal{O}^i_{vd}$ with middle leg in column $-$ corresponds to a unique path system $\mathcal{P}'\in\mathcal{O}^i_{vd}$ with middle leg in column $+$.  
\begin{center}
\begin{table}[h]
\begin{tabular}{|c|c|p{1in}|}
\hline
- & +\\
\hline
$\xymatrixrowsep{.5pc}\xymatrixcolsep{1pc}\xymatrix{
&&*=0{\bullet}\ar@{-}[rd]&&*=0{\bullet}\ar@{.}[rd]\ar@{.}[ld]\ar@{-}[d]&&*=0{\circ}\ar@{.}[ld]&&\\
&&&*=0{\bullet}\ar@{-}[rd]\ar@{.}[d]&*=0{\bullet}\ar@{.}[rd]\ar@{-}[ld]&*=0{\circ}\ar@{.}[ld]\ar@{.}[d]&&&\\
&&&*=0{\bullet}\ar@{.}[rd]\ar@{-}[ld]&*=0{\bullet}\ar@{-}[d]&*=0{\circ}\ar@{.}[ld]\ar@{.}[rd]&&&\\
&&*=0{\bullet}&&*=0{\bullet}&&*=0{\circ}&&\\}$ &
$\xymatrixrowsep{.5pc}\xymatrixcolsep{1pc}\xymatrix{
&&*=0{\bullet}\ar@{-}[rd]&&*=0{\bullet}\ar@{-}[rd]\ar@{.}[ld]\ar@{.}[d]&&*=0{\circ}\ar@{.}[ld]&&\\
&&&*=0{\bullet}\ar@{.}[rd]\ar@{-}[d]&*=0{\circ}\ar@{.}[rd]\ar@{.}[ld]&*=0{\bullet}\ar@{-}[ld]\ar@{.}[d]&&&\\
&&&*=0{\bullet}\ar@{.}[rd]\ar@{-}[ld]&*=0{\bullet}\ar@{-}[d]&*=0{\circ}\ar@{.}[ld]\ar@{.}[rd]&&&\\
&&*=0{\bullet}&&*=0{\bullet}&&*=0{\circ}&&\\}$\\
\hline
$\xymatrixrowsep{.5pc}\xymatrixcolsep{1pc}\xymatrix{
&&*=0{\circ}\ar@{.}[rd]&&*=0{\bullet}\ar@{.}[rd]\ar@{.}[ld]\ar@{-}[d]&&*=0{\bullet}\ar@{-}[ld]&&\\
&&&*=0{\circ}\ar@{.}[rd]\ar@{.}[d]&*=0{\bullet}\ar@{-}[rd]\ar@{.}[ld]&*=0{\bullet}\ar@{-}[ld]\ar@{.}[d]&&&\\
&&&*=0{\circ}\ar@{.}[rd]\ar@{.}[ld]&*=0{\bullet}\ar@{-}[d]&*=0{\bullet}\ar@{.}[ld]\ar@{-}[rd]&&&\\
&&*=0{\circ}&&*=0{\bullet}&&*=0{\bullet}&&\\}$ &
$\xymatrixrowsep{.5pc}\xymatrixcolsep{1pc}\xymatrix{
&&*=0{\circ}\ar@{.}[rd]&&*=0{\bullet}\ar@{.}[rd]\ar@{-}[ld]\ar@{.}[d]&&*=0{\bullet}\ar@{-}[ld]&&\\
&&&*=0{\bullet}\ar@{-}[rd]\ar@{.}[d]&*=0{\circ}\ar@{.}[rd]\ar@{.}[ld]&*=0{\bullet}\ar@{.}[ld]\ar@{-}[d]&&&\\
&&&*=0{\circ}\ar@{.}[rd]\ar@{.}[ld]&*=0{\bullet}\ar@{-}[d]&*=0{\bullet}\ar@{.}[ld]\ar@{-}[rd]&&&\\
&&*=0{\circ}&&*=0{\bullet}&&*=0{\bullet}&&\\}$\\
\hline
\end{tabular}
\caption{middle legs of a path system}
\label{tab:E7}
\end{table}
\end{center}

Therefore we can write \eqref{eq:GV} as
$$\sum_{\mathcal{P}\in\mathcal{O}}\sgn(\mathcal{P})\omega(\mathcal{P})$$
where $\mathcal{O}$ is the subset of path systems in $\mathcal{O}^i_{vd}$ whose middle legs do \emph{not} appear in Table \ref{tab:E7}.  It is straightforward to check that the path sytems in $\mathcal{O}$ all have the same sign.  Therefore Lemma \ref{lem:easy} implies that \eqref{eq:GV} is non-zero for $5\leq i\leq 8$.  This completes the proof of Proposition \ref{prop:FEE}.
\end{proof}

To deal with the remaining types, we appeal to the symmetry of $G$.
There is an \emph{antiautomorphism} of the graph $G$ defined by 
\begin{equation}
\label{eq:symmetric2}
\alpha\colon x\mapsto w_0\cdot x\cdot w_0(\Theta)
\end{equation} 
where $w_0$ is the longest word in $W$ and $w_0(\Theta)$ is the longest word in $W_\Theta$; see \cite{BB} Proposition 2.5.4.  The map $\alpha$ induces a linear identification (which we denote by the same name) 
\begin{equation}
\label{eq:transpose}
\xymatrixcolsep{4pc}\xymatrixrowsep{1pc}\xymatrix{\left(\rel\right)^i\ar[r]^-{\alpha} & \left(\left(\rel\right)^{r-i}\right)^*\\
X_u\ar@{|->}[r] & \bar{A}_{\alpha(u)}\\}
\end{equation}  
where $\bar{A}_{\alpha(u)}$ denotes the BGG-operator restricted to $\left(\rel\right)^{r-i}$.
 
Let $\chi\in\rel$ be any homogeneous element of degree $1$, regarded as an operator of degree $1$ on $\rel$ given by multiplication by $\chi$.  Let $\chi^*$ denote the adjoint operator of degree $(-1)$.
\begin{proposition}
\label{prop:symmetric}
For each $0\leq i\leq r$ the diagram 
\begin{equation}
\label{eq:symmetric}
\xymatrixcolsep{5pc}\xymatrixrowsep{4pc}\xymatrix{\left(\rel\right)^i\ar[r]^-{\chi}\ar[d]_-{\alpha} & \left(\rel\right)^{i+1}\ar[d]^-{\alpha}\\
\left(\left(\rel\right)^{r-i}\right)^*\ar[r]_-{\chi^*} & \left(\left(\rel\right)^{r-i-1}\right)^*\\}
\end{equation}
commutes.
\end{proposition}
\begin{proof}
To see that \eqref{eq:symmetric} commutes it suffices to show for each $\ell(u)=i$ and $\ell(v)=i+1$ that 
\begin{equation}
\label{eq:symm2}
A_v\left(\chi\cdot X_u\right)=A_{\alpha(u)}\left(\chi\cdot X_{\alpha(v)}\right).
\end{equation}
The LHS of \eqref{eq:symm2} is given by 
$$A_v\left(\chi\cdot X_u\right)=\check{\beta}(u(\chi))$$
where $\xyto{u\ar[r]^-{\beta} & v\\}$, and the RHS is given by
$$A_{\alpha(u)}\left(\chi\cdot X_{\alpha(v)}\right)=\check{\gamma}(\alpha(v)(\chi))$$
where $\xyto{\alpha(v)\ar[r]^-{\gamma} & \alpha(u)\\}$.  It is straightforward to verify that $\gamma=-w_0(\beta)$ and hence that $\check{\gamma}(\alpha(v)(\chi))=\check{\beta}(u(\chi))$ as desired.
\end{proof}

\begin{remark*}
It turns out that the pairing of $\left(\rel\right)^i$ with $\left(\rel\right)^{r-i}$ given by $\alpha$ in \ref{eq:transpose} agrees with the ``intersection pairing" on $S_W$ given by multiplication.  In other words,
$$\left(A,B\right)\coloneqq\alpha(B)\left(A\right)=A\cdot B.$$
See \cite{Hil} Theorem 2.9 page 147.  
\end{remark*}

Proposition \ref{prop:symmetric} implies that the Lefschetz matrix \eqref{eq:matrix} is symmetric.  In fact we see that the matrix \eqref{eq:matrix} has the form $A^t\cdot B\cdot A$ as follows:  let $A$ be the matrix for the ``first leg" of the Lefschetz map $$\bar{\rho}^{\lfloor\frac{r}{2}\rfloor-i}\colon \left(\rel\right)^i\rightarrow\left(\rel\right)^{\lfloor\frac{r}{2}\rfloor}$$
and let $B$ be the matrix for the ``second leg" 
$$\begin{cases}
\bar{\rho}\colon \left(\rel\right)^{\lfloor\frac{r}{2}\rfloor}\rightarrow\left(\rel\right)^{\lceil\frac{r}{2}\rceil} & \text{if $r$ is odd}\\
\alpha\colon\left(\rel\right)^{\frac{r}{2}}\rightarrow\left(\rel\right)^{\frac{r}{2}} & \text{if $r$ is even}.\\
\end{cases}$$
Then Proposition \ref{prop:symmetric} implies that the matrix for the ``third leg" of the Lefschetz map
$$\bar{\rho}^{\lfloor\frac{r}{2}\rfloor-i}\colon \left(\rel\right)^{\lceil\frac{r}{2}\rceil}\rightarrow\left(\rel\right)^{r-i}$$
can be identified with $A^t$, the transpose of the ``first leg", using $\alpha$ to identify $\left(\rel\right)^{r-i}$ with $\left(\left(\rel\right)^i\right)^*\cong\left(\rel\right)^i$ for each $0\leq i\leq\lfloor\frac{r}{2}\rfloor$.

Now suppose the matrix $B$ above is positive definite, and that $\rel$ has the following \emph{weak} Lefschetz property:  the multiplication map
$\bar{\rho}\colon\left(\rel\right)^i\rightarrow\left(\rel\right)^{i+1}$ is injective for $0\leq i<\lfloor\frac{r}{2}\rfloor$.
Then clearly the matrix $A^t\cdot B\cdot A$ must also be positive definite.  This handy observation reduces our task of checking that $\rel$ has the strong Lefschetz property to checking that $\rel$ has the weak Lefschetz property, which is computationally much simpler to carry out.

\begin{lemma}
\label{prop:E8H4}
Given $W$ of type $E_8$ or $H_4$, there is a choice of parabolic subgroup $W_\Theta\subset W$ such that $\rel$ having the strong Lefschetz property is equivalent to $\rel$ having the weak Lefschetz property.
\end{lemma}
\begin{proof}
For $W$ in type $E_8$ (resp.\ $H_4$) choose $W_\Theta\subset W$ a parabolic subgroup of type $E_7$ (resp.\ $H_3$).  Table \ref{tab:E8H4} below shows the Bruhat order on $W^\Theta$ in middle degrees (i.e. $28\sra 29$ (on top) in type $E_8$ and $22\sra 23$ (on top) in type $H_4$).  The unlabelled edges in the graphs in Table \ref{tab:E8H4} have weight $1$.  A straightforward computation shows that these matrices are positive definite, hence the result of Lemma \ref{prop:E8H4} follows from our discussion following Proposition \ref{prop:symmetric}.
\end{proof}

\begin{remark*}
We observe that the matrix in Table \ref{tab:E8H4} for type $E_8$ resembles the Cartan matrix for the root system of type $E_8$.  In fact one can show that it is \emph{similar} to the Cartan matrix; the change of basis just changes the sign of every other simple root i.e.,
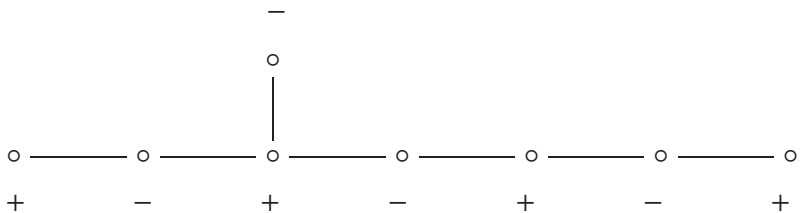
\begin{align*}
\xymatrixcolsep{3.45pc}\xymatrix{&&-&&&&\\}\\
\xymatrixcolsep{3pc}\xymatrix{&&\circ&&&&\\
\circ\ar@{-}[r]&\circ\ar@{-}[r]&\circ\ar@{-}[r]\ar@{-}[u]& \circ\ar@{-}[r]&\circ\ar@{-}[r]&\circ\ar@{-}[r]&\circ}\\
\xymatrixcolsep{2.8pc}\xymatrix{
+ & - & + & - & + & - & +}.\\ 
\end{align*}
In particular this shows that the matrix in Table \ref{tab:E8H4} for type $E_8$ is positive definite.  We thank Tom Braden for pointing out this neat proof.
\end{remark*}

\begin{center}
\begin{table}[h]
\begin{tabular}{|c|c|c|}
\hline
$H_4$ & $\begin{array}{c}\xymatrixcolsep{3pc}\xymatrixrowsep{4pc}\xymatrix{
*=0{\bullet}\ar@{-}[d]|{2}\ar@{-}[rd] & *=0{\bullet}\ar@{-}[d]|{2}\ar@{-}[ld]\ar@{-}[rd] & *=0{\bullet}\ar@{-}[d]|{2}\ar@{-}[ld]\ar@{-}[rd]|<<<<<{2b} & *=0{\bullet}\ar@{-}[ld]|<<<<<{2b}\ar@{-}[d]|{2}\\
*=0{\bullet} & *=0{\bullet} & *=0{\bullet} & *=0{\bullet}\\}\\
\end{array}$ &
$\begin{array}{c}
\left(\begin{array}{cccc}
2 & 1 & 0 & 0\\
1 & 2 & 1 & 0\\
0 & 1 & 2 & 2b\\
0 & 0 & 2b & 2\\
\end{array}\right)\\
\Big[b=\frac{-1+\sqrt{5}}{4}\Big]\\
\end{array}$\\ 
\hline

$E_8$ & $\begin{array}{c}\xymatrixcolsep{3pc}\xymatrixrowsep{4pc}\xymatrix{*=0{\bullet}\ar@{-}[d]|{2}\ar@{-}[rrrrrrrd] & *=0{\bullet}\ar@{-}[d]|{2}\ar@{-}[rd]\ar@{-}[rrrd] & *=0{\bullet}\ar@{-}[ld]\ar@{-}[d]|{2}\ar@{-}[rrrrrd] & *=0{\bullet}\ar@{-}[d]|{2}\ar@{-}[rd] & *=0{\bullet}\ar@{-}[llld]\ar@{-}[ld]\ar@{-}[d]|{2}\ar@{-}[rd] & *=0{\bullet}\ar@{-}[ld]\ar@{-}[d]|{2}\ar@{-}[rd] & *=0{\bullet}\ar@{-}[d]|{2}\ar@{-}[ld] & *=0{\bullet}\ar@{-}[d]|{2}\ar@{-}[llllld]\ar@{-}[llllllld]\\
*=0{\bullet} & *=0{\bullet} & *=0{\bullet} & *=0{\bullet} & *=0{\bullet} & *=0{\bullet} & *=0{\bullet} & *=0{\bullet}\\}\\
\end{array}$ & 
$\left(\begin{array}{cccccccc}
2 & 0 & 0 & 0 & 0 & 0 & 0 & 1\\
0 & 2 & 1 & 0 & 1 & 0 & 0 & 0\\
0 & 1 & 2 & 0 & 0 & 0 & 0 & 1\\
0 & 0 & 0 & 2 & 1 & 0 & 0 & 0\\
0 & 1 & 0 & 1 & 2 & 1 & 0 & 0\\
0 & 0 & 0 & 0 & 1 & 2 & 1 & 0\\
0 & 0 & 0 & 0 & 0 & 1 & 2 & 0\\
1 & 0 & 1 & 0 & 0 & 0 & 0 & 2\\
\end{array}\right)$\\

\hline
\end{tabular}
\caption{$H_4$ and $E_8$ in middle degrees}
\label{tab:E8H4} 
\end{table}
\end{center}

\begin{proposition}
\label{prop:H4E8lef}
For $W$ of type $H_4$ or $E_8$, there exists a parabolic subgroup $W_\Theta\subset W$ such that $\rel$ has the strong Lefschetz property.
\end{proposition}
\begin{proof}
By Proposition \ref{prop:GV} and Lemma \ref{prop:E8H4} we need only check that the weighted path matrices in $G$ from $V_i$ to $V_{i+1}$ have full rank for $0\leq i\leq \lfloor\frac{r}{2}\rfloor$.  Given below are tables showing the non-trivial bipartite graphs in each of the types $H_4$ and $E_8$.  The second and fifth columns show the bipartite graphs between $V_{i-1}$ and $V_i$ (on top), where $i$ is listed in the adjacent column (to the left):  the empty circles with dashed lines indicate that we are ignoring that corresponding row in our weighted path matrix to exhibit a maximal non-singular submatrix.  The numbers in the third and sixth columns enumerate the path systems in the bipartite graph to the left of the entry:  the symbol ``($\setminus\circ$)" indicates that we are not including the empty circle and dashed lines in our count, and the symbol ``(same $\pm$)" indicates that all of the enumerated path systems have the same sign.    

Table \ref{tab:H4} shows the directed bipartite graphs in type $H_4$ giving the matrices for the Lefschetz maps $\bar{\rho}\colon\left(\rel\right)^i\rightarrow\left(\rel\right)^{i+1}$ for $7\leq i\leq 22$.  Note that the matrices with corresponding entries $1$, $1$ ($\setminus\circ$), or 2 (same $\pm$) must have full rank and the computation of their edge weights is not required.  Those entries decorated with a $\ast$ can be verified by direct computation:  The weights for these bipartite graphs have been computed using the root system given in \cite{Hum} page 47, where $a=\frac{1+\sqrt{5}}{4}$ and $b=\frac{-1+\sqrt{5}}{4}$ (as before, the unlabelled edges (in the starred cases) are assumed to have weight $1$).  

\begin{table}
\begin{center}
\begin{tabular}{|c|r|c||c|r|c|p{1in}|}
\hline
22 & $\xyGgraph{*=0{\bullet}\ar@{-}[d] & *=0{\bullet}\ar@{-}[d]\ar@{-}[ld] & *=0{\bullet}\ar@{-}[ld]\ar@{-}[d]|{2a}\ar@{-}[rd]|<<<<{2b} &*=0{\bullet}\ar@{-}[d]|{2a}\ar@{-}[ld]|<<<<{2b}\\
*=0{\bullet} & *=0{\bullet} & *=0{\bullet} & *=0{\bullet}\\}
$ & 2 $\ast$ & 14 & $\xyGraph{& *=0{\bullet}\ar@{-}[rd]\ar@{-}[rrd]& *=0{\bullet}\ar@{-}[ld] & *=0{\bullet}\ar@{-}[d]\ar@{-}[lld]\\
& *=0{\bullet} & *=0{\bullet} & *=0{\bullet}\\}
$ & 1\\
&&&&&\\
\hline
21 & $\xyGraph{*=0{\bullet}\ar@{-}[d]& *=0{\bullet}\ar@{-}[d]\ar@{-}[ld]\ar@{-}[rrd]& *=0{\bullet}\ar@{-}[ld]\ar@{-}[d] & *=0{\bullet}\ar@{-}[d]\ar@{-}[ld]\\
*=0{\bullet} & *=0{\bullet} & *=0{\bullet} & *=0{\bullet}\\}
$ & 2 (same $\pm$) & 13 & $\xyGraph{& *=0{\bullet}\ar@{-}[rd]& *=0{\bullet}\ar@{-}[ld]\ar@{-}[rd] & *=0{\bullet}\ar@{-}[d]\ar@{-}[ld]\\
& *=0{\bullet} & *=0{\bullet} & *=0{\bullet}\\}
$ & 1 \\
&&&&&\\
\hline
20 & $\xyGgraph{*=0{\bullet}\ar@{-}[rd]|<<<<{2a}\ar@{-}[rrrd]|>>>>{2b}& *=0{\bullet}\ar@{-}[d]\ar@{-}[ld]|>>>>>{2a}\ar@{-}[rd]|>>>>{2b}& *=0{\bullet}\ar@{-}[lld]\ar@{-}[d]|<<<<{2a} & *=0{\bullet}\ar@{-}[d]|{2a}\ar@{-}[ld]|{2b}\\
*=0{\bullet} & *=0{\bullet} & *=0{\bullet} & *=0{\bullet}\\}
$& 3 $\ast$ & 12 & $\xyGgraph{& *=0{\bullet}\ar@{-}[d]|{2a}\ar@{-}[rrd]|<<<<{2b}& *=0{\bullet}\ar@{-}[d]|>>>>{2a}\ar@{-}[rd] & *=0{\bullet}\ar@{-}[d]|{2a}\ar@{-}[ld]|<<<<<{2b}\ar@{-}[lld]\\
& *=0{\bullet} & *=0{\bullet} & *=0{\bullet}\\}
$ & 3 $\ast$\\
&&&&&\\
\hline
19 & $\xyGraph{*=0{\bullet}\ar@{-}[d]\ar@{-}[rd]\ar@{-}[rrd]& *=0{\bullet}\ar@{-}[d]\ar@{-}[rrd]& *=0{\bullet}\ar@{-}[rd]\ar@{-}[d] & *=0{\bullet}\ar@{-}[d]\\
*=0{\bullet} & *=0{\bullet} & *=0{\bullet} & *=0{\bullet}\\}
$& 1 & 11 & $\xyGraph{& *=0{\bullet}\ar@{-}[rd]& *=0{\bullet}\ar@{-}[ld]\ar@{-}[rd] & *=0{\bullet}\ar@{-}[d]\ar@{-}[ld]\\
& *=0{\bullet} & *=0{\bullet} & *=0{\bullet}\\}
$ & 1\\
&&&&&\\
\hline
18 & $\xyGraph{*=0{\bullet}\ar@{-}[d]\ar@{-}[rd]& *=0{\bullet}\ar@{-}[rd]\ar@{-}[ld]\ar@{-}[rrd]& *=0{\bullet}\ar@{-}[ld]\ar@{-}[d]\ar@{-}[rd] & *=0{\bullet}\ar@{-}[d]\\
*=0{\bullet} & *=0{\bullet} & *=0{\bullet} & *=0{\bullet}\\}
$& 2 (same $\pm$) & 10 & $\xyGraph{& *=0{\circ}\ar@{--}[rrd]& *=0{\bullet}\ar@{-}[d]& *=0{\bullet}\ar@{-}[d]\ar@{-}[ld]\\
&& *=0{\bullet} & *=0{\bullet}\\}
$ & 1 ($\setminus\circ$)\\
&&&&&\\
\hline
17 & $\xyGraph{*=0{\bullet}\ar@{-}[rd]\ar@{-}[rrd]& *=0{\bullet}\ar@{-}[ld]\ar@{-}[rd]& *=0{\bullet}\ar@{-}[ld]\ar@{-}[rd] & *=0{\bullet}\ar@{-}[d]\ar@{-}[ld]\\
*=0{\bullet} & *=0{\bullet} & *=0{\bullet} & *=0{\bullet}\\}
$ & 2 (same $\pm$) & 9 & $\xyGraph{& & *=0{\bullet}\ar@{-}[d]& *=0{\bullet}\ar@{-}[d]\ar@{-}[ld]\\
&& *=0{\bullet} & *=0{\bullet}\\}
$ & 1\\
&&&&&\\
\hline
16 & $\xyGraph{*=0{\circ}\ar@{--}[rrd]& *=0{\bullet}\ar@{-}[d]\ar@{-}[rrd]& *=0{\bullet}\ar@{-}[rd]\ar@{-}[d] & *=0{\bullet}\ar@{-}[d]\\
& *=0{\bullet} & *=0{\bullet} & *=0{\bullet}\\}
$ & 1 & 8 & $\xyGraph{& & *=0{\bullet}\ar@{-}[d]& *=0{\bullet}\ar@{-}[d]\ar@{-}[ld]\\
&& *=0{\bullet} & *=0{\bullet}\\}
$ & 1\\
&&&&&\\
\hline
15 & $\xyGraph{& *=0{\bullet}\ar@{-}[d]\ar@{-}[rrd]& *=0{\bullet}\ar@{-}[d] & *=0{\bullet}\ar@{-}[d]\ar@{-}[ld]\\
& *=0{\bullet} & *=0{\bullet} & *=0{\bullet}\\}
$ & 1 & 7 & $\xyGraph{& & *=0{\bullet}\ar@{-}[d]\ar@{-}[rd]& *=0{\bullet}\ar@{-}[d]\\
&& *=0{\bullet} & *=0{\bullet}\\}
$ & 1 \\
&&&&&\\
\hline
\end{tabular}

\end{center}
\caption{weak Lefschetz property for $H_4\big/H_3$}
\label{tab:H4}
\end{table}

Table \ref{fig:E8} shows the directed bipartite graphs in type $E_8$ giving the matrices for the Lefschetz maps $\bar{\rho}\colon\left(\rel\right)^i\rightarrow\left(\rel\right)^{i+1}$ for $7\leq i\leq 28$.  Again, the matrices with corresponding entries $1$, $1$ ($\setminus\circ$), or $2$ (same $\pm$) must have full rank.  It turns out that \emph{all} of the edges appearing in Table \ref{fig:E8} have weight $1$, as we show presently.  This implies that the matrices corresponding to \emph{odd} entries also have full rank.

\begin{table}
\begin{center}
\begin{tabular}{|c|r|c||c|r|c|p{1in}|}
\hline
28 & $\xyGraph{*=0{\circ}\ar@{--}[rrrrrrd] & *=0{\bullet}\ar@{-}[d]\ar@{-}[rd] & *=0{\bullet}\ar@{-}[ld]\ar@{-}[rrrrd] & *=0{\bullet}\ar@{-}[d] & *=0{\bullet}\ar@{-}[lld]\ar@{-}[ld]\ar@{-}[d] & *=0{\bullet}\ar@{-}[ld]\ar@{-}[d] & *=0{\bullet}\ar@{-}[ld] & *=0{\bullet}\ar@{-}[d]\ar@{-}[ld]\\
& *=0{\bullet} & *=0{\bullet} & *=0{\bullet} & *=0{\bullet} & *=0{\bullet} & *=0{\bullet} & *=0{\bullet}\\}
$ & 1 ($\setminus\circ$) & 17 & $\xyGraph{*=0{\bullet}\ar@{-}[d] & *=0{\bullet}\ar@{-}[rd]\ar@{-}[d] & *=0{\bullet}\ar@{-}[lld]\ar@{-}[d] & *=0{\bullet}\ar@{-}[d]\ar@{-}[lld] & *=0{\bullet}\ar@{-}[d]\ar@{-}[ld]\ar@{-}[lld]\\
*=0{\bullet} & *=0{\bullet} & *=0{\bullet} & *=0{\bullet} & *=0{\bullet}\\}
$ & 1\\
&&&&&\\
\hline
27 & $\xyGraph{*=0{\bullet}\ar@{-}[rrd]\ar@{-}[rrrrrrd] & *=0{\bullet}\ar@{-}[d]\ar@{-}[rd]\ar@{-}[rrd] & *=0{\bullet}\ar@{-}[ld]\ar@{-}[rrd] & *=0{\bullet}\ar@{-}[d]\ar@{-}[rd]\ar@{-}[rrd] & *=0{\bullet}\ar@{-}[rd] & *=0{\bullet}\ar@{-}[llllld] & *=0{\bullet}\ar@{-}[lllllld]\ar@{-}[d]\\
*=0{\bullet} & *=0{\bullet} & *=0{\bullet} & *=0{\bullet} & *=0{\bullet} & *=0{\bullet} & *=0{\bullet}\\}
$ & 2 (same $\pm$) & 16 & $\xyGraph{*=0{\bullet}\ar@{-}[rd] & *=0{\bullet}\ar@{-}[rd] & *=0{\circ}\ar@{--}[ld]\ar@{--}[d]\ar@{--}[rd] & *=0{\bullet}\ar@{-}[rd]\ar@{-}[ld] & *=0{\bullet}\ar@{-}[d]\ar@{-}[ld]\\
& *=0{\bullet} & *=0{\bullet} & *=0{\bullet} & *=0{\bullet}\\}
$ & 1 ($\setminus\circ$)\\
&&&&&\\
\hline
26 & $\xyGraph{*=0{\bullet}\ar@{-}[d] & *=0{\bullet}\ar@{-}[d]\ar@{-}[rd] & *=0{\bullet}\ar@{-}[ld]\ar@{-}[rrrrd]\ar@{-}[rd] & *=0{\bullet}\ar@{-}[d]\ar@{-}[rd]\ar@{-}[ld] & *=0{\bullet}\ar@{-}[rd]\ar@{-}[lld] & *=0{\bullet}\ar@{-}[ld]\ar@{-}[d] & *=0{\bullet}\ar@{-}[lllllld]\ar@{-}[d]\\
*=0{\bullet} & *=0{\bullet} & *=0{\bullet} & *=0{\bullet} & *=0{\bullet} & *=0{\bullet} & *=0{\bullet}\\}
$& 3 & 15 & $\xyGraph{*=0{\bullet}\ar@{-}[rd]\ar@{-}[d] & *=0{\bullet}\ar@{-}[ld]\ar@{-}[rd] & *=0{\bullet}\ar@{-}[d]\ar@{-}[ld] & *=0{\bullet}\ar@{-}[d]\ar@{-}[ld]\\
*=0{\bullet} & *=0{\bullet} & *=0{\bullet} & *=0{\bullet}\\}
$ & 2 (same $\pm$)\\
&&&&&\\
\hline
25 & $\xyGraph{*=0{\bullet}\ar@{-}[d] & *=0{\bullet}\ar@{-}[rd]\ar@{-}[rrrd] & *=0{\bullet}\ar@{-}[ld]\ar@{-}[d]\ar@{-}[rd] & *=0{\bullet}\ar@{-}[rrrd]\ar@{-}[rd]\ar@{-}[ld] & *=0{\bullet}\ar@{-}[d]\ar@{-}[ld] & *=0{\bullet}\ar@{-}[lld] & *=0{\bullet}\ar@{-}[lllllld]\ar@{-}[d]\ar@{-}[ld]\\
*=0{\bullet} & *=0{\bullet} & *=0{\bullet} & *=0{\bullet} & *=0{\bullet} & *=0{\bullet} & *=0{\bullet}\\}
$& 2 (same $\pm$) & 14 & $\xyGraph{*=0{\bullet}\ar@{-}[rd] & *=0{\bullet}\ar@{-}[ld]\ar@{-}[d] & *=0{\bullet}\ar@{-}[d]\ar@{-}[ld] & *=0{\bullet}\ar@{-}[d]\ar@{-}[ld]\\
*=0{\bullet} & *=0{\bullet} & *=0{\bullet} & *=0{\bullet}\\}
$ & 1\\
&&&&&\\
\hline
24 & $\xyGraph{*=0{\bullet}\ar@{-}[d]\ar@{-}[rd] & *=0{\bullet}\ar@{-}[rd]\ar@{-}[rrd] & *=0{\bullet}\ar@{-}[d]\ar@{-}[rrd]\ar@{-}[rrrd] & *=0{\bullet}\ar@{-}[d]\ar@{-}[rd] & *=0{\bullet}\ar@{-}[d]\ar@{-}[rrd] & *=0{\bullet}\ar@{-}[d]\ar@{-}[llllld] & *=0{\bullet}\ar@{-}[llllld]\ar@{-}[d]\ar@{-}[ld]\\
*=0{\bullet} & *=0{\bullet} & *=0{\bullet} & *=0{\bullet} & *=0{\bullet} & *=0{\bullet} & *=0{\bullet}\\}
$& 5 & 13 & $\xyGraph{*=0{\bullet}\ar@{-}[d] & *=0{\bullet}\ar@{-}[ld]\ar@{-}[d] & *=0{\bullet}\ar@{-}[d]\ar@{-}[ld] & *=0{\bullet}\ar@{-}[d]\ar@{-}[ld]\\
*=0{\bullet} & *=0{\bullet} & *=0{\bullet} & *=0{\bullet}\\}
$ & 1\\
&&&&&\\
\hline
23 & $\xyGraph{*=0{\bullet}\ar@{-}[d] & *=0{\bullet}\ar@{-}[ld]\ar@{-}[d] & *=0{\bullet}\ar@{-}[rd]\ar@{-}[rrd]\ar@{-}[rrrd] & *=0{\bullet}\ar@{-}[ld]\ar@{-}[rd] & *=0{\bullet}\ar@{-}[d]\ar@{-}[rrd] & *=0{\bullet}\ar@{-}[d]\ar@{-}[llllld]\ar@{-}[rd] & *=0{\bullet}\ar@{-}[llllld]\ar@{-}[d]\\
*=0{\bullet} & *=0{\bullet} & *=0{\bullet} & *=0{\bullet} & *=0{\bullet} & *=0{\bullet} & *=0{\bullet}\\}
$ & 1 & 12 & $\xyGraph{*=0{\circ}\ar@{--}[rd] & *=0{\bullet}\ar@{-}[rd]\ar@{-}[d] & *=0{\bullet}\ar@{-}[d]\ar@{-}[rd] & *=0{\bullet}\ar@{-}[d]\\
& *=0{\bullet} & *=0{\bullet} & *=0{\bullet}\\}
$ & 1 ($\setminus\circ$)\\
&&&&&\\
\hline
22 & $\xyGraph{*=0{\bullet}\ar@{-}[rd]\ar@{-}[rrd] & *=0{\bullet}\ar@{-}[rd] & *=0{\bullet}\ar@{-}[rd] & *=0{\bullet}\ar@{-}[rd]\ar@{-}[rrd] & *=0{\circ}\ar@{--}[ld]\ar@{--}[rrd]\ar@{--}[d] & *=0{\bullet}\ar@{-}[d]\ar@{-}[lllld]\ar@{-}[rd] & *=0{\bullet}\ar@{-}[lllld]\ar@{-}[d]\\
& *=0{\bullet} & *=0{\bullet} & *=0{\bullet} & *=0{\bullet} & *=0{\bullet} & *=0{\bullet}\\}
$ & 1 ($\setminus\circ$) & 11 & $\xyGraph{*=0{\bullet}\ar@{-}[rd]\ar@{-}[d] & *=0{\bullet}\ar@{-}[d]\ar@{-}[rd] & *=0{\bullet}\ar@{-}[d]\\
*=0{\bullet} & *=0{\bullet} & *=0{\bullet}\\}
$ & 1\\
&&&&&\\
\hline
21 & $\xyGraph{*=0{\bullet}\ar@{-}[rd]\ar@{-}[d] & *=0{\bullet}\ar@{-}[d] & *=0{\bullet}\ar@{-}[d]\ar@{-}[rd] & *=0{\bullet}\ar@{-}[ld]\ar@{-}[rd] & *=0{\bullet}\ar@{-}[d]\ar@{-}[lllld]\ar@{-}[rd] & *=0{\bullet}\ar@{-}[lllld]\ar@{-}[ld]\ar@{-}[lld]\\
*=0{\bullet} & *=0{\bullet} & *=0{\bullet} & *=0{\bullet} & *=0{\bullet} & *=0{\bullet}\\}
$& 2 (same $\pm$)  & 10 & $\xyGraph{*=0{\circ}\ar@{--}[rd] & *=0{\bullet}\ar@{-}[d]\ar@{-}[rd] & *=0{\bullet}\ar@{-}[d]\\
& *=0{\bullet} & *=0{\bullet}\\}
$ & 1 ($\setminus\circ$)\\
&&&&&\\
\hline
20 & $\xyGraph{*=0{\bullet}\ar@{-}[rd]\ar@{-}[rrrd] & *=0{\bullet}\ar@{-}[d]\ar@{-}[ld] & *=0{\bullet}\ar@{-}[d]\ar@{-}[rrd] & *=0{\bullet}\ar@{-}[llld]\ar@{-}[rd] & *=0{\bullet}\ar@{-}[d]\ar@{-}[llld]\ar@{-}[rd] & *=0{\bullet}\ar@{-}[d]\ar@{-}[lld]\\
*=0{\bullet} & *=0{\bullet} & *=0{\bullet} & *=0{\bullet} & *=0{\bullet} & *=0{\bullet}\\}
$& 3 & 9 & $\xyGraph{*=0{\bullet}\ar@{-}[d]\ar@{-}[rd] & *=0{\bullet}\ar@{-}[d]\\
*=0{\bullet} & *=0{\bullet}\\}
$ & 1\\
&&&&&\\
\hline
 19 & $\xyGraph{*=0{\bullet}\ar@{-}[rd] & *=0{\bullet}\ar@{-}[d]\ar@{-}[rrd] & *=0{\bullet}\ar@{-}[d]\ar@{-}[rrd] & *=0{\bullet}\ar@{-}[llld]\ar@{-}[d] & *=0{\bullet}\ar@{-}[d]\ar@{-}[llld]\ar@{-}[rd] & *=0{\bullet}\ar@{-}[d]\ar@{-}[lld]\\
*=0{\bullet} & *=0{\bullet} & *=0{\bullet} & *=0{\bullet} & *=0{\bullet} & *=0{\bullet}\\}
$& 1 & 8 & $\xyGraph{*=0{\bullet}\ar@{-}[d]\ar@{-}[rd] & *=0{\bullet}\ar@{-}[d]\\
*=0{\bullet} & *=0{\bullet}\\}
$ & 1\\
&&&&&\\
\hline
18 & $\xyGraph{*=0{\bullet}\ar@{-}[rd] & *=0{\bullet}\ar@{-}[rd]\ar@{-}[rrd] & *=0{\bullet}\ar@{-}[rrd] & *=0{\bullet}\ar@{-}[lld]\ar@{-}[d] & *=0{\circ}\ar@{--}[d]\ar@{--}[lld]\ar@{--}[rd] & *=0{\bullet}\ar@{-}[d]\ar@{-}[lld]\\
& *=0{\bullet} & *=0{\bullet} & *=0{\bullet} & *=0{\bullet} & *=0{\bullet}\\}
$ & 1 ($\setminus\circ$) & 7 & $\xyGraph{*=0{\bullet}\ar@{-}[d]\ar@{-}[rd] & *=0{\bullet}\ar@{-}[d]\\
*=0{\bullet} & *=0{\bullet}\\}
$ & 1\\
&&&&&\\
\hline
\end{tabular}
\end{center}
\caption{weak Lefschetz property for $E_8\big/E_7$}
\label{fig:E8}
\end{table}

Using the root system (and simple system) of type $E_8$ described in \cite{Hum} page 43, we compute $\bar{\rho}$ to be a multiple of the longest root $\theta$.  Hence after a rescaling we can take our Lefschetz element to be $\theta$.  It is straightforward to verify that  
$$\check{\gamma}(\theta)=\begin{cases}
2 & \text{if} \ \ \gamma=\theta\\
1 & \text{otherwise}\\
\end{cases}$$
for all $\gamma\in\Phi^+\setminus\Phi^+_\Theta$.  Hence the only weights appearing on the directed edges of $G$ are $1$ and $2$.  
\begin{claim*}
The weight $2$ appears only in the middle degree.
\end{claim*}  
To verify the claim, we define the \emph{height} of a root $$h(\beta)\coloneqq\sum_{\gamma\in\Delta}c_\gamma$$
where $\beta=\sum_{\gamma\in\Delta}c_\gamma\cdot\gamma$.  Now suppose that $\xyto{\bar{w}\ar[r]^-\alpha & \bar{w}'}$ for some $\bar{w}$, $\bar{w}'\in W^\Theta$ and some $\alpha\in\Phi^+$ such that $\bar{w}^{-1}(\alpha)=\theta$.  Then $\alpha=\bar{w}(\theta)\in\Phi^+$ and a standard argument reveals that $0<h(\alpha)<h(\theta)-\ell(\bar{w})$.  Since $h(\theta)=29$ we conclude that $\ell(\bar{w})<29$.  By the symmetry of the Bruhat order on $W^\Theta$ given by \eqref{eq:symmetric2}, we conclude that $\ell(\bar{w})=28$ which establishes the claim.  

This completes the proof of Proposition \ref{prop:H4E8lef}.
\end{proof}

\section{Proof of Theorem \ref{thm:coinvlef}}
We are now ready to put it all together.  Given any parabolic subgroup $W_\Theta\subseteq W$, there is a natural inclusion
$$\pi\colon S^{W_\Theta}_W\rightarrow S_W$$
with respect to which $S_W$ is a free $\rel$-module by Proposition \ref{thm:coinvlh}.
There is a natural surjective ring homomorphism
$$\iota\colon S_W\rightarrow S_{W_\Theta}$$
whose kernel is the ideal $(\rel)^+\cdot S_W$ by Proposition \ref{thm:coinvlh}.  Thus if $S_{W_\Theta}$ and $\rel$ both have the strong Lefschetz property, then Theorem \ref{thm:fibun} implies that $S_W$ also has the strong Lefschetz property.  

\begin{proof}[Proof of Theorem \ref{thm:coinvlef}]
We first show it for the infinite families (i.e. classical and dihedral types).  
Assume $W$ is of type $A$, $B$, $D$ or $I_2(m)$.  We argue by induction on the rank of $W$ (i.e. $\dim_\R\left(\spa\left\{\Phi\right\}\right)$).  The base case is trivial (i.e. $W=\{e\}$ and $\Phi=\emptyset$).  Assume the assertion holds for finite reflection groups of rank $<n$ and let $W$ be a finite reflection group of rank $n$ and of one of the above types.  Then by Proposition \ref{prop:altypeslefschetz}, there exists a parabolic subgroup $W_\Theta\subset W$ such that $S_W^{W_\Theta}$ has the strong Lefschetz property.  By the induction hypothesis, $S_{W_\Theta}$ also has the strong Lefschetz property.  Therefore by Theorem \ref{thm:fibun}, $S_W$ must have the strong Lefschetz property and we are done by induction.

We can now inductively build on this result for the infinite families to get the result for the remaining types.

If $W$ is of type $H_3$ (resp.\ $F_4$, $E_6$) then by Proposition \ref{prop:altypeslefschetz} (resp.\ Proposition \ref{prop:FEE}), there is a parabolic subgroup $W_\Theta\subset W$ of type $I_2(5)$ (resp.\ $B_3$, $D_5$) such that $S_W^{W_\Theta}$ has the strong Lefschetz property.  By the preceding argument, $S_{W_\Theta}$ also has the strong Lefschetz property.  Therefore by Theorem \ref{thm:fibun}, $S_W$ must also have the strong Lefschetz property.

If $W$ is of type $E_7$ then by Proposition \ref{prop:FEE} there is a parabolic subgroup $W_\Theta\subset W$ of type $E_6$ such that $\rel$ has the strong Lefschetz property.  By the preceding argument $S_{W_\Theta}$ also has the strong Lefschetz property.  Therefore by Theorem \ref{thm:fibun}, $S_W$  must also have the strong Lefschetz property.

If $W$ is of type $E_8$ (resp.\ $H_4$) then by Proposition \ref{prop:H4E8lef} there is a parabolic subgroup $W_\Theta\subset W$ of type $E_7$ (resp.\ $H_3$) such that $S_W^{W_\Theta}$ has the strong Lefschetz property.  By the preceding arguments the coinvariant ring $S_{W_\Theta}$ also has the strong Lefschetz property.  Therefore by Theorem \ref{thm:fibun}, $S_W$ must also have the strong Lefschetz property. 
This completes the proof of Theorem \ref{thm:coinvlef}.
\end{proof}

\section{Concluding Remarks}
During the preparation of this manuscript, I discovered that Proctor \cite{P2} also gives a proof of the strong Lefschetz property for $\rel$ (with respect to certain maximal parabolic subgroups $W_\Theta\subset W$) in type $A_n$ using weighted path sums and Proposition \ref{prop:GV}.  

A biproduct of Theorem \ref{thm:coinvlef} is that the relative coinvariant rings $\rel$ also have the strong Lefschetz property with respect to \emph{some} parabolic subgroup $W_\Theta\subset W$.  It would be nice to know if this holds for \emph{any} parabolic subgroup.    

In proving Proposition \ref{prop:H4E8lef}, we have actually proved that the relative coinvariant rings for the quotients $E_8\big/E_7$ and $H_4\big/H_3$ not only have the strong Lefschetz property, but satisfy the stronger \emph{Hodge-Riemann bilinear relations}.  In the crystallographic cases this is expected since the (relative) coinvariant rings are cohomology rings of smooth projective varieties.  In the $H_4\big/H_3$ case, this result seems to be new.  It would be interesting to know if the Hodge-Riemann bilinear relations hold for all coinvariant rings.  

It would also be interesting to see to what extent Theorem \ref{thm:coinvlef} extends to coinvariant rings of complex finite reflection groups.
\bibliographystyle{plain}
\bibliography{coinvariant}
\end{document}